\documentclass{article}
\usepackage{amssymb,amsmath,amscd}
\usepackage[amsmath,thmmarks]{ntheorem}
\theoremseparator{.}
\theoremsymbol{\rule{1ex}{1ex}}
\theorembodyfont{\upshape}
\newtheorem{definition}{Definition}[section]
\newtheorem{remark}[definition]{Remark}

\newtheorem{example}[definition]{Example}

\newtheorem{question}[definition]{Question}

\newtheorem*{proof}{Proof}

\newtheorem{setup}[definition]{}
\theorembodyfont{\itshape}

\newtheorem{Theorem}[definition]{Theorem}

\theoremnumbering{Alph}
\newtheorem{inttheorem}{Theorem}

\newtheorem*{warning}{Warning}
\theoremsymbol{}
\newtheorem{lemma}[definition]{Lemma}
\newtheorem{proposition}[definition]{Proposition}
\newtheorem{theorem}[definition]{Theorem}
\newtheorem{corollary}[definition]{Corollary}
\usepackage[a4paper,hscale=0.8,vscale=0.8,hmarginratio=1:1,includehead,headheight=14pt]{geometry}
\usepackage{fancyhdr}
\usepackage{amssymb,amsmath,amscd,dsfont}
\usepackage[all,cmtip,line]{xy}
\usepackage{slashed}
\usepackage{colonequals}

\newenvironment{tabsection}{}{}

\newcommand{\op}[1]{\ensuremath{\operatorname{#1}}}
\newcommand{\wt}[1]{\ensuremath{\widetilde{#1}}}
\newcommand{\wh}[1]{\ensuremath{\widehat{#1}}}

\newcommand{\cA}{\ensuremath{\mathcal{A}}}

\newcommand{\cG}{\ensuremath{\mathcal{G}}}
\newcommand{\cH}{\ensuremath{\mathcal{H}}}

\newcommand{\cK}{\ensuremath{\mathcal{K}}}

\newcommand{\cN}{\ensuremath{\mathcal{N}}}

\newcommand{\cP}{\ensuremath{\mathcal{P}}}

\newcommand{\fh}{\ensuremath{\mathfrak{h}}}

\newcommand{\mf}[1]{\ensuremath{\mathfrak{#1}}}

 \newcommand{\R}{\ensuremath{\mathbb{R}}}
 
 \newcommand{\N}{\ensuremath{\mathbb{N}}}

\newcommand{\id}{\ensuremath{\operatorname{id}}}
\newcommand{\pr}{\ensuremath{\operatorname{pr}}}
\newcommand{\ev}{\ensuremath{\operatorname{ev}}}

\newcommand{\Ad}{\ensuremath{\operatorname{Ad}}}

\newcommand{\supp}{\ensuremath{\operatorname{supp}}}

\newcommand{\cat}[1]{\ensuremath{\mathsf{\mathop{#1}}}}

\newcommand{\Aut}{\ensuremath{\operatorname{Aut}}}

\newcommand{\Diff}{\ensuremath{\operatorname{Diff}}}

\newcommand{\Gau}{\ensuremath{\operatorname{Gau}}}

\newcommand{\im}{\ensuremath{\operatorname{im}}}

\newcommand{\res}{\ensuremath{\operatorname{res}}}

\newcommand{\se}{\ensuremath{\nobreak\subseteq\nobreak}}
\newcommand{\from}{\ensuremath{\nobreak\colon\nobreak}}
\renewcommand{\to}{\ensuremath{\nobreak\rightarrow\nobreak}}
\newcommand{\toto}{\ensuremath{\nobreak\rightrightarrows\nobreak}}

\newcommand{\coloneq}{\colonequals}
\DeclareMathOperator{\A}{\Sigma}
\newcommand{\norm}[1]{\left\lVert #1 \right\rVert}
\newcommand{\opnorm}[1]{\norm{#1}_\text{op}}
\DeclareMathOperator{\Fl}{Fl}
\newcommand{\SectRI}[1]{\ensuremath{\Gamma^{\rho} (T^\alpha #1 )}}

\newcommand{\vectL}[1]{\Gamma^{\lambda} (#1) }
\newcommand{\LB}[1][\cdot \hspace{1pt} , \cdot]{\left[\hspace{1pt} #1 \hspace{1pt} \right]}
\newcommand{\BoundOp}[1]{\mathcal{L}\left(#1\right)}
\DeclareMathOperator{\evol}{evol}
\DeclareMathOperator{\Evol}{Evol}
\DeclareMathOperator{\one}{{\bf 1}}

\newcommand\opn{\ensuremath{\mathrel{\mathpalette\opncls\circ}}}

\newcommand{\opncls}[2]{
  \ooalign{$#1\subseteq$\cr
  \hidewidth\raisefix{#1}\hbox{$#1{\stylefix{#1}#2}\mkern2mu$}\cr}}
\def\raisefix#1{
  \ifx#1\displaystyle
    \raise.39ex
  \else
    \ifx#1\textstyle
      \raise.39ex
    \else
      \ifx#1\scriptstyle
        \raise.275ex
      \else
        \raise.150ex
      \fi
    \fi
  \fi
}
\def\stylefix#1{
  \ifx#1\displaystyle
    \scriptstyle
  \else
    \ifx#1\textstyle
      \scriptstyle
    \else
      \ifx#1\scriptstyle
        \scriptscriptstyle
      \else
        \scriptscriptstyle
      \fi
    \fi
  \fi
}
\DeclareFontFamily{U}{mathx}{\hyphenchar\font45}
\DeclareFontShape{U}{mathx}{m}{n}{
      <5> <6> <7> <8> <9> <10>
      <10.95> <12> <14.4> <17.28> <20.74> <24.88>
      mathx10
      }{}
\DeclareSymbolFont{mathx}{U}{mathx}{m}{n}
\DeclareFontSubstitution{U}{mathx}{m}{n}
\DeclareMathAccent{\widecheck}{0}{mathx}{"71}
\DeclareMathAccent{\wideparen}{0}{mathx}{"75}
\usepackage{graphicx}%
\usepackage[all,cmtip,line]{xy}%
\CompileMatrices%
\usepackage[notcite,notref,final]{showkeys}%
\usepackage[normalem]{ulem}
\usepackage{enumitem}%
\setlist[enumerate]{label={\alph*})}%
\usepackage{hyperref}%
\hypersetup{
 urlcolor = red,
 colorlinks = true,
 linkcolor = blue,
 citecolor = blue,
 linktocpage = true,
 pdftitle = {The Lie group of bisections of a Lie groupoid},
 pdfauthor = {Alexander Schmeding, Christoph Wockel},
 bookmarksopen = true,
 bookmarksopenlevel = 1,
 unicode = true,
 hypertexnames =false 
}%
\newcommand{\Bis}{\ensuremath{\op{Bis}}}
\newcommand{\Lf}{\ensuremath{\mathbf{L}}}

\newcommand{\eqclass}[1]{\ensuremath{[#1]}}

\begin{document}

\begin{flushright}
   {\sf ZMP-HH/14-18}\\
   {\sf Hamburger$\;$Beitr\"age$\;$zur$\;$Mathematik$\;$Nr.$\;$523}\\[2mm]
\end{flushright}

\title{The Lie group of bisections of a Lie groupoid} \author{Alexander
Schmeding\footnote{NTNU Trondheim, Norway
\href{mailto:alexander.schmeding@math.ntnu.no}{alexander.schmeding@math.ntnu.no}
}%
~~and Christoph Wockel\footnote{University of Hamburg, Germany
\href{mailto:christoph@wockel.eu}{christoph@wockel.eu}}}
{\let\newpage\relax\maketitle}

\begin{abstract}
 In this article we endow the group of bisections of a Lie groupoid with
 compact base with a natural locally convex Lie group structure. Moreover, we
 develop thoroughly the connection to the algebra of sections of the associated
 Lie algebroid and show for a large class of Lie groupoids that their groups of
 bisections are regular in the sense of Milnor.
\end{abstract}

\medskip


\medskip


\tableofcontents

\section*{Introduction}

\begin{tabsection}
 Infinite-dimensional and higher structures are amongst the important concepts
 in modern Lie theory. This comprises homotopical and higher Lie algebras
 ($L_{\infty}$-algebras) and Lie groups (group stacks and Kan simplicial
 manifolds), Lie algebroids, Lie groupoids and generalisations thereof (e.g.,
 Courant algebroids) and infinite-dimensional locally convex Lie groups and Lie
 algebras. This paper is a contribution to the heart of Lie theory in the sense
 that is connects two regimes, namely the theory of Lie groupoids, Lie
 algebroids and infinite-dimensional Lie groups and Lie algebras. This
 connection is established by associating to a Lie groupoid its group of
 bisections and establishing a locally convex Lie group structure on this
 group.
 
 The underlying idea is not new per se, statements like ``...the group of
 (local) bisections is a (local) Lie group whose Lie algebra is given by the
 sections of the associated Lie algebroid...'' can be found at many places in
 the literature. In fact it depends on the setting of generalised manifolds
 that one uses, whether or not this statement is a triviality or a theorem. For
 instance, if the category of smooth spaces in which one works is cartesian
 closed and has finite limits, then the bisections are automatically a group
 object in this category and the only difficulty might be to calculate its Lie
 algebra. This applies for instance to diffeological spaces, where it follows
 from elementary theory that the bisections of a diffeological groupoid are
 naturally a diffeological group\footnote{A natural diffeology on the
 bisections of a diffeological groupoid would be the subspace diffeology of the
 functional diffeology on the space of smooth maps from the objects to the
 arrows.}. Another such setting comes from (higher) smooth topoi. See for
 instance
 \cite{Schreiber13Differential-cohomology-in-a-cohesive-infinity-topos,FiorenzaRogersSchreiber13Higher-geometric-prequantum-theory}
 for a generalisation of bisections to higher groupoids and
 \cite{Domenico-Fiorenza13L-infinity-algebras-of-local-observables-from-higher-prequantum-bundles}
 for a construction of the corresponding infinitesimal $L_{\infty}$-algebra.
 Moreover, in synthetic differential geometry the derivation of the Lie algebra
 of the group of bisections can also be done formally
 \cite{Nishimura06The-Lie-algebra-of-the-group-of-bisections}.
 
 What we aim for in this paper is a natural locally convex Lie group structure
 on the group of bisections, which is not covered by the settings and
 approaches mentioned above. What comes closest to this aim are the results
 from \cite{Rybicki02A-Lie-group-structure-on-strict-groups}, where a group
 structure in the ``convenient setting of global analysis'' is established.
 However, the results of the present paper are much stronger and more general
 than the ones from \cite{Rybicki02A-Lie-group-structure-on-strict-groups} in
 various respects, which we now line out. First of all, we work throughout in
 the locally convex setting
 \cite{neeb2006,hg2002a,Milnor84Remarks-on-infinite-dimensional-Lie-groups} for
 infinite-dimensional manifolds. The locally convex setting has the advantage
 that it is compatible with the underlying topological framework. In
 particular, smooth maps and differentials of those are automatically
 continuous. This will become important in the geometric applications that we
 have in mind (work in progress). Secondly, we not only construct a Lie group
 structure on the bisections, but also relate it to (and in fact derive it
 from) the canonical smooth structure on manifolds of mappings. Thus one is
 able to identify many naturally occurring maps as smooth maps. For instance,
 the natural action of the bisections on the arrow manifold is smooth, which
 allows for an elegant identification of the Lie bracket on the associated Lie
 algebra. The latter then gives rise to a \emph{natural} isomorphism between
 the functors that naturally arise in this context, namely the (bi)section
 functors and the Lie functors. This is the third important feature of this
 paper. The last contribution of this paper is that we prove that the
 bisections are in fact a regular Lie group for all Banach-Lie groupoids.
 
 On the debit side, one should say that the exhaustive usage of smooth
 structures on mapping spaces forces us to work throughout with locally
 metrisable manifolds and over compact bases, although parts of our results
 should be valid in greater generality. Moreover, the proof of regularity is
 quite technical, which is the reason for deferring several details of it to a
 separate section. To say it once more, the results are not surprising in any
 respect, it is the coherence of all these concepts that is the biggest value
 of the paper. \bigskip
 
 We now go into some more detail and explain the main results. Suppose
 $\cG = (G \toto M)$ is a Lie groupoid. This means that $G,M$ are smooth
 manifolds, equipped with submersions $\alpha,\beta\from G\to M$ and an
 associative and smooth multiplication $G\times _{\alpha,\beta}G\to G$ that
 admits a smooth identity map $1\from M\to G$ and a smooth inversion
 $\iota\from G\to G$. Then the bisections $\Bis(\cG)$ of $\cG$ are the sections
 $\sigma\from M\to G$ of $\alpha$ such that $\beta \circ \sigma$ is a
 diffeomorphism of $M$. This becomes a group with respect to
 \begin{equation*}
  (\sigma \star \tau ) (x) \coloneq \sigma ((\beta \circ \tau)(x))\tau(x)\text{ for }  x \in M.
 \end{equation*}
 Our main tool to construct a Lie group structure on the group of bisections
 are certain local additions on the space of arrows $G$. This is generally the
 tool one needs on the target manifold to understand smooth structures on
 mapping spaces (see
 \cite{michor1980,conv1997,Wockel13Infinite-dimensional-and-higher-structures-in-differential-geometry}
 or Appendix \ref{Appendix: MFD}). We require that the local addition on $G$ is
 adapted to the source projection $\alpha$, i.e.\ it restricts to a local
 addition on each fibre $\alpha^{-1} (x)$ for $ x \in M$. If the groupoid $\cG$
 admits such an addition, we deduce the following (Theorem \ref{theorem: A}):
 
 \begin{inttheorem}\label{thm:inttheorem_a}
  Suppose $\cG = (G \toto M)$ is a locally convex and locally metrisable Lie
  groupoid with $M$ compact. If $G$ admits an adapted local addition, then the
  group $\Bis (\cG)$ is a submanifold of $C^\infty (M,G)$. With this structure,
  $\Bis (\cG)$ is a locally convex Lie group modelled on a metrisable space.
 \end{inttheorem}
 
 After having constructed the Lie group structure on $\Bis (\cG)$ we show that
 a large variety of Lie groupoids admit adapted local additions, including all
 finite-dimensional Lie groupoids, all Banach Lie-groupoids with smoothly
 paracompact $M$ and all locally trivial Lie groupoids with locally exponential
 vertex group.
 
 We then determine the Lie algebra associated to $\Bis(\cG)$. Our investigation
 shows that the Lie algebra is closely connected to the Lie algebroid
 associated to the Lie groupoid. Explicitly, Theorem \ref{theorem: LAlg} may be
 subsumed as follows.
 
 \begin{inttheorem}
  Suppose $\cG$ is a Lie groupoid which satisfies the assumptions of Theorem
  \ref{thm:inttheorem_a}. Then the Lie algebra of the Lie group $\Bis (\cG)$ is
  naturally isomorphic (as a topological Lie algebra) to the Lie algebra of
  sections of the Lie algebroid associated to $\cG$, endowed with the negative
  of the usual bracket.
 \end{inttheorem}
 
 After this we briefly discuss some perspectives for further research. We then
 investigate regularity properties of the Lie group $\Bis (\cG)$. To this end,
 recall the notion of regularity for Lie groups:
 
 Let $H$ be a Lie group modelled on a locally convex space, with identity
 element $\one$, and $r\in \N_0\cup\{\infty\}$. We use the tangent map of the
 right translation $\rho_h\colon H\to H$, $x\mapsto xh$ by $h\in H$ to define
 $v.h\coloneq T_{\one} \rho_h(v) \in T_h H$ for $v\in T_{\one} (H) =: L(H)$.
 Following \cite{Dahmen2012} and \cite{hg2012}, $H$ is called
 \emph{$C^r$-regular} if for each $C^r$-curve
 $\gamma\colon [0,1]\rightarrow L(H)$ the initial value problem
 \begin{equation}\label{eq: regular}
  \begin{cases}
  \eta'(t)&= \gamma(t).\eta(t)\\ \eta(0) &= \one
  \end{cases}
 \end{equation}
 has a (necessarily unique) $C^{r+1}$-solution
 $\Evol (\gamma)\coloneq\eta\colon [0,1]\rightarrow H$, and the map
 \begin{displaymath}
  \evol \colon C^r([0,1],L(H))\rightarrow H,\quad \gamma\mapsto \Evol
  (\gamma)(1)
 \end{displaymath}
 is smooth. If $H$ is $C^r$-regular and $r\leq s$, then $H$ is also
 $C^s$-regular. A $C^\infty$-regular Lie group $H$ is called \emph{regular}
 \emph{(in the sense of Milnor}) -- a property first defined in
 \cite{Milnor84Remarks-on-infinite-dimensional-Lie-groups}. Every finite
 dimensional Lie group is $C^0$-regular (cf. \cite{neeb2006}). Several
 important results in infinite-dimensional Lie theory are only available for
 regular Lie groups (see
 \cite{Milnor84Remarks-on-infinite-dimensional-Lie-groups}, \cite{neeb2006},
 \cite{hg2012}, cf.\ also \cite{conv1997} and the references therein). We prove
 the following result (Theorems \ref{thm:ltriv_reg} and \ref{theorem:
 Banach:reg}):
 
 \begin{inttheorem}
  Let $\cG = (G \toto M)$ be a Lie groupoid that admits a local addition and has
  compact space of objects $M$. Suppose either that $G$ is a Banach-manifold or
  that $\cG$ is locally trivial with locally exponential and $C^{k}$-regular
  vertex group. Then the Lie group $\Bis (\cG)$ is $C^k$-regular for each
  $k \in \N_0 \cup \{\infty\}$. In particular, the group $\Bis (\cG)$ is
  regular in the sense of Milnor.
 \end{inttheorem}
\end{tabsection}

\begin{tabsection}
 Note that all assumptions that we will impose throughout this paper are
 satisfied for finite-dimensional Lie groupoids over compact manifolds. For
 this case, the above theorems may be subsumed as follows:
\end{tabsection}

\begin{inttheorem}
 If $\cG=(G\toto M)$ is a finite-dimensional Lie groupoid with compact $M$,
 then $\Bis(\cG)$ is a regular Fr\'echet-Lie group modelled on the space of
 sections $\Gamma(\Lf(\cG))$ of the  Lie algebroid $\Lf(\cG)$. Moreover, the Lie
 bracket on $\Gamma(\Lf(\cG))$ induced from the Lie group structure on
 $\Bis(\cG)$ is the negative of the Lie bracket underlying $\Lf(\cG)$.
\end{inttheorem}

\section{Locally convex Lie groupoids and Lie groups}
\label{sec:locally_convex_lie_groupoids_and_lie_groups}

\begin{tabsection}
 In this section we recall the Lie theoretic notions and conventions that we
 are using in this paper. We refer to
 \cite{Mackenzie05General-theory-of-Lie-groupoids-and-Lie-algebroids} for an
 introduction to (finite-dimensional) Lie groupoids and the associated group of
 bisections. The notation for Lie groupoids and their structural maps also
 follows \cite{Mackenzie05General-theory-of-Lie-groupoids-and-Lie-algebroids}.
 However, we do not restrict our attention to finite dimensional Lie groupoids.
 Hence, we have to augment the usual definitions with several comments. Note
 that we will work all the time over a fixed base manifold $M$.
\end{tabsection}

\begin{setup}
 Let $\cG = (G \toto M)$ be a groupoid over $M$ with source projection
 $\alpha \colon G \rightarrow M$ and target projection
 $\beta \colon G \rightarrow M$. Then $\cG$ is a \emph{(locally convex and
 locally metrisable) Lie groupoid over $M$}\footnote{See Appendix
 \ref{Appendix: MFD} for references on differential calculus in locally convex
 spaces.} if
 \begin{itemize}
  \item the objects $M$ and the arrows $G$ are locally convex and locally
        metrisable manifolds,
  \item the smooth structure turns $\alpha$ and $\beta$ into surjective
        submersions, i.e., they are locally projections\footnote{This implies
        in particular that the occurring fibre-products are submanifolds of the
        direct products, see \cite[Appendix
        C]{Wockel13Infinite-dimensional-and-higher-structures-in-differential-geometry}.}
  \item the partial multiplication
        $m \colon G \times_{\alpha,\beta} G \rightarrow G$, the object
        inclusion $1 \colon M \rightarrow G$ and the inversion
        $\iota \colon G \rightarrow G$ are smooth.
 \end{itemize}
 The \emph{group of bisections} $\Bis (\cG)$ of $\cG$ is given as the set of sections
 $\sigma \colon M \rightarrow G$ of $\alpha$ such that
 $\beta \circ \sigma \colon M \rightarrow M$ is a diffeomorphism. This is a
 group with respect to
 \begin{equation}\label{eq: BISGP1}
  (\sigma \star \tau ) (x) \coloneq \sigma ((\beta \circ \tau)(x))\tau(x)\text{ for }  x \in M.
 \end{equation}
 The object inclusion $1 \colon M \rightarrow G$ is then the neutral element
 and the inverse element of $\sigma$ is
 \begin{equation}\label{eq: BISGP2}
  \sigma^{-1} (x) \coloneq \iota( \sigma ((\beta \circ\sigma)^{-1} (x)))\text{ for } x \in M.
 \end{equation}
\end{setup}

\begin{remark}\label{rem:structure_on_bisections}
 \begin{enumerate}
  \item The definition of bisection is not symmetric with respect to $\alpha$
        and $\beta$. This lack of symmetry can be avoided by defining a
        bisection as a set (see \cite[p.
        23]{Mackenzie05General-theory-of-Lie-groupoids-and-Lie-algebroids}).
        This point of view is important for instance in Poisson geometry, where
        one wants to restrict the image of bisection to be Lagrangian
        submanifolds in a symplectic groupoid
        \cite{Rybicki01On-the-group-of-Lagrangian-bisections-of-a-symplectic-groupoid,Xu97Flux-homomorphism-on-symplectic-groupoids}.
        However, we will not need this point of view in the present article.
  \item Each bisection $\sigma$ gives rise to a \emph{left-translation}
        $L_\sigma \colon G \rightarrow G, g \mapsto \sigma (\beta (g))g$. The
        map
        \begin{displaymath}
         (\Bis (\cG), \star) \rightarrow (\Diff (G),\circ), \sigma \mapsto
         L_\sigma
        \end{displaymath}
        induces a group isomorphism onto the subgroup of all left translations
        (cf.\ \cite[p.\
        22]{Mackenzie05General-theory-of-Lie-groupoids-and-Lie-algebroids}).
        Similarly we could identify the bisections with right translations on
        $G$
  \item \label{rem:structure_on_bisections_c} The group of bisections naturally
        acts on the arrows by
        \begin{equation*}
         \gamma \colon \Bis (\cG) \times G \rightarrow G , (\psi , g) \mapsto
         L_\psi (g) = \psi (\beta (g)) g
        \end{equation*}
        (in fact, the group structure on $\Bis(\cG)$ is derived from this
        action, see \cite[\S
        1.4]{Mackenzie05General-theory-of-Lie-groupoids-and-Lie-algebroids} or
        \cite[\S
        15.3]{Cannas-da-SilvaWeinstein99Geometric-models-for-noncommutative-algebras}).
        This action will play an important r\^ole when computing the Lie
        algebra of the Lie group of bisections in Section
        \ref{sec:the_lie_algebra_of_the_group_of_bisections}.
  \item The group $\Bis(\cG)$ depends functorially on $\cG$ in the following
        way. Suppose $\cH=(H\toto M)$ is another Lie groupoid over $M$ and that
        $f\from \cG\to \cH$ is a morphism of Lie groupoids over $M$, i.e., it
        is a smooth functor $f\from G\to H$ which is the identity on the
        objects $f \circ 1_{\cG}=1_{\cH}$. Then there is an induced morphism of
        the groups of bisections
        \begin{equation*}
         \Bis(f)\from \Bis(\cG)\to \Bis(\cH),\quad \sigma\mapsto f\circ\sigma.
        \end{equation*}
        If $\cK$ is another Lie groupoid over $M$ and $g\from \cH\to \cK$
        another morphism, then we clearly have
        $\Bis(g \circ f)=\Bis(g) \circ \Bis(f)$. Lie groupoids over $M$,
        together with their morphisms form a category $\cat{LieGroupoids}_{M}$.
        We can thus interpret $\Bis$ as a functor
        \begin{equation}\label{eqn:functor_1}
         \Bis\from \cat{LieGroupoids}_{M}\to \cat{Groups}.
        \end{equation}
  \item \label{rem: tangentgpd}From $\cG=(G\toto M)$ we can construct a new Lie
        groupoid $T\cG\coloneq (TG\toto TM)$. This has the surjective
        submersions $T \alpha$ and $T \beta$ as source and target projections,
        $T1$ as object inclusion and $T \iota$ as inversion.
        
        In order to define the multiplication we first have to identify
        $T(G\times_{\alpha,\beta}G)$ with $TG\times_{T \alpha,T \beta}G$. To
        this end we first recall that $G\times_{\alpha,\beta}G$ is the
        submanifold $\{(a,b)\in G\times G\mid \alpha(a)=\beta(b)\}$ of
        $G\times G$. We may thus identify $T(G\times_{\alpha,\beta}G)$ via the
        isomorphism $T(G\times G)\cong TG\times TG$ with a subset of
        $TG\times TG$. Now we claim that
        \begin{equation}\label{eqn1}
         T(G\times_{\alpha,\beta}G)=\{(x,y)\in TG\times TG\mid T \alpha(x)=T \beta(y) \}=TG\times_{T \alpha,T \beta}TG
        \end{equation}
        as subsets (and thus as submanifolds) of $TG\times TG$. Note that the
        statement is a local one (we just have to find representing smooth
        curves), meaning that we may assume $G$ and $M$ to be diffeomorphic to
        open subsets $U\opn X$ and $V\opn Y$ for locally convex spaces $X$ and
        $Y$. Since $\alpha$ and $\beta$ are submersions we can also assume that
        $X=Z\times Y$, that $U=W\times V$, that $\alpha=\beta=\pr_{2}$ and that
        there are diffeomorphisms $\varphi\from W\times V\to W\times V$ and
        $\psi\from V\to V$ that make
        \begin{equation*}
         \vcenter{\xymatrix{
         G\ar[d]^{\alpha}\ar[r]^(.4){\cong} &  W\times V\ar[r]^{\varphi}\ar[d]^{\pr_{2}} & W\times V\ar[d]^{\pr_{2}}&G\ar[l]_(.4){\cong}\ar[d]^{\beta}\\
         M\ar[r]^{\cong} &  V\ar[r]^{\psi} & V & M\ar[l]_{\cong}
         }}
        \end{equation*}
        commute. In particular, we have
        $\varphi(w,v)=(\varphi_{1}(w,v),\psi(v))$. By composing $\varphi$ with
        the diffeomorphism $(x,y)\mapsto(x,\psi^{-1}(y))$ we may assume that
        $\psi$ is the identity. For the inner square we then have
        \begin{equation*}
         T(W\times V\times_{\pr_{2},\pr_{2}} W\times V)=\{((w,z,v,y),(w',z',v',y'))\in (W\times Z\times V\times Y)^{2}\mid v=v' \text{ and }y=y' \}.
        \end{equation*}
        Since $T\pr_{2}=\pr_{2}\times \pr_{2}$ one sees that for the inner
        square we also have
        \begin{equation*}
         T(W\times V)\times _{T\pr_{2},T\pr_{2}} T(W\times V)=
         \{((w,z,v,y),(w',z',v',y'))\in (W\times Z\times V\times Y)^{2}\mid v=v'\text{ and }y=y'\}.
        \end{equation*}
        This shows that both sides in \eqref{eqn1} are actually the same. We
        thus may set
        \begin{equation*}
         T \mu\from TG\times_{T \alpha,T \beta}TG\to G
        \end{equation*}
        with respect this identification. One can easily check that yields in
        fact a new Lie groupoid $T\cG$.
 \end{enumerate}
\end{remark}

\begin{tabsection}
 We now recall the construction of the Lie algebroid associated to a Lie
 groupoid.
\end{tabsection}

\begin{setup}
 \label{setup: alpha:sbd} We consider the subset
 $T^\alpha G = \bigcup_{g\in G} T_g \alpha^{-1} \alpha (g)$ of $TG$. 
 Note that  for all $x \in T^\alpha_g G$ the definition implies
 $T\alpha (x) = 0_{\alpha (g)} \in T_{\alpha (g)} M$, i.e.\ fibre-wise we have
 $T^\alpha_g G = \ker T_g\alpha$. Since $\alpha$ is a submersion, the
 same is true for $T\alpha$. Computing in submersion charts, the kernel of
 $T_g \alpha$ is a direct summand of the model space of $TG$. Furthermore, the
 submersion charts of $T \alpha$ yield submanifold charts for $T^\alpha G$
 whence $T^\alpha G$ becomes a split submanifold of $T G$. Restricting the
 projection of $TG$, we thus obtain a subbundle
 $\pi_\alpha \colon T^\alpha G \to G$ of the tangent bundle $TG$.
\end{setup}

\begin{setup}\label{setup:
 algebroid} We now recall the construction of the Lie algebroid $\Lf(\cG)$
 associated to a Lie groupoid $\cG$. The vector bundle underlying $\Lf(\cG)$ is
 the pullback $1^{*}T^{\alpha}G$ of the bundle $T^\alpha G$ via the embedding
 $1 \colon M \rightarrow G$. We denote this bundle also by $\Lf(G)\to M$. The
 anchor $a_{\Lf(\cG)} \colon \Lf(G) \rightarrow TM$ is the composite of the
 morphisms
 \begin{displaymath}
  \Lf(G) \rightarrow T^\alpha G \xrightarrow{\subseteq} TG \xrightarrow{T\beta}
  TM
 \end{displaymath}
 To describe the Lie bracket on $\Gamma (\Lf(G))$ we need some notation: Let
 $g$ be an element of $G$. We define the smooth map
 $R_g \colon \alpha^{-1}(\beta (g)) \rightarrow G$, $h \mapsto hg$. A vertical
 vector field $Y \in \Gamma (T^{\alpha}G)$ is called \emph{right-invariant} if
 for all $(h,g) \in G \times_{\alpha, \beta} G$ the equation
 $Y(hg) = T_h (R_g)(Y(h))$ holds. We denote the Lie subalgebra of all right
 invariant vector fields on $G$ by $\SectRI{G}$. Then \cite[Corollary
 3.5.4]{Mackenzie05General-theory-of-Lie-groupoids-and-Lie-algebroids} shows
 that the assignment
 \begin{equation}\label{eq: RI:VF}
  \Gamma(\Lf(\cG)) \rightarrow \SectRI{G},\quad X \mapsto \overrightarrow{X} , \quad \text{ with } \quad\overrightarrow{X} (g) = T(R_g) (X(\beta (g)))
 \end{equation}
 is an isomorphism of $C^\infty (G)$-modules. Its inverse is given by
 $\SectRI{G} \rightarrow \Gamma (\Lf(\cG))$, $X \mapsto X \circ 1$. Now we
 define the Lie bracket on $\Gamma (\Lf(\cG))$ via
 \begin{equation}\label{eq: LB}
  \LB[X,Y] \coloneq \LB[\overrightarrow{X},\overrightarrow{Y}] \circ 1.
 \end{equation}
 Then the \emph{Lie algebroid $\Lf (\cG)$ of $\cG$} is the vector bundle
 $\Lf(G) \rightarrow M$ together with the bracket $\LB$ from \eqref{eq: LB} and
 the anchor $a_{\Lf (\cG)}$.
\end{setup}

\begin{setup}
 To fully describe the Lie functor on Lie groupoids, suppose that
 $\cH=(H\toto M)$ is another Lie groupoid over M and that $f\from G\to H$ is a
 smooth functor satisfying $f \circ 1_{\cG}=1_{\cH}$. Then
 $Tf(T^{\alpha}G)\se T^{\alpha}H$ and from $f \circ 1_{\cG}=1_{\cH}$ it follows
 that $Tf$ induces a morphism
 $1_{\cG}^{*}T^{\alpha}G\to 1_{\cH}^{*}T^{\alpha}H$ of vector bundles. This
 morphism is in fact a morphism of Lie algebroids \cite[Proposition
 3.5.10]{Mackenzie05General-theory-of-Lie-groupoids-and-Lie-algebroids}, which
 we denote by $\Lf(f)\from \Lf(\cG)\to \Lf(\cH)$. In total, this defines the
 Lie functor
 \begin{equation*}
  \Lf\from \cat{LieGroupoids}_{M}\to \cat{LieAlgebroids}_{M}.
 \end{equation*}
\end{setup}

\begin{tabsection}
 We now turn to the Lie functor defined on the category of locally convex Lie
 groups (cf.\
 \cite{neeb2006,Milnor84Remarks-on-infinite-dimensional-Lie-groups}).
\end{tabsection}

\begin{setup}\label{setup:
 RI:VF} Let $H$ be a locally convex Lie group, i.e., a locally convex manifold
 which is a group such that the group operations are smooth. The Lie algebra
 $\Lf(H)$ of $H$ is the tangent space $T_1 H$ endowed with a suitable Lie
 bracket $\LB{}$ (cf.\ \cite[Definition II1.5]{neeb2006}, \cite[\S
 5]{Milnor84Remarks-on-infinite-dimensional-Lie-groups}). To obtain the
 bracket, we identify $T_1 H$ with the Lie algebra of left invariant vector
 fields $\vectL{H}$. Each element $X \in T_1 H$ extends to a (unique) left
 invariant vector field
 \begin{displaymath}
  X^\lambda \in \Gamma (TH)\quad \text{ via }\quad X^\lambda (h) = T_1 \lambda _h (X).
 \end{displaymath}
 Here $\lambda _h$ is the left translation in $H$ by the element $h$.
 Similarly, to $X$ there corresponds a unique right invariant vector field
 $X^{\rho}$. Since the bracket of left invariant vector fields is left
 invariant and the bracket of right invariant vector fields is right invariant
 there are now two ways of endowing $T_{1}H$ with a Lie bracket. The convention
 here is to define the bracket $T_{1}H$ via \emph{left} invariant vector field.
 Thus $X\mapsto X^{\lambda}$ becomes an isomorphism of Lie algebras and
 $X\mapsto X^{\rho}$ becomes an anti-isomorphism of Lie algebras, i.e., we have
 $-[X,Y]^{\rho}=[X^{\rho},Y^{\rho}]$ (\cite[Assertion
 5.6]{Milnor84Remarks-on-infinite-dimensional-Lie-groups}).
\end{setup}

\begin{setup}
 Suppose $H,H'$ are locally convex Lie groups and $f\from H\to H'$ is a smooth
 group homomorphism. Then $T_{1}f\from T_{1}H\to T_{1}H'$ is a continuous and
 linear map which preserves the Lie bracket. This defines the morphism
 $\Lf(f)\from \Lf(H)\to \Lf(H')$ of topological Lie algebras associated to $f$.
 In total, we obtain this way the Lie functor
 \begin{equation*}
  \Lf\from \cat{LieGroups}\to\cat{LieAlgebras}.
 \end{equation*}
\end{setup}

\begin{warning}
 Each Lie group $H$ gives rise to a Lie groupoid $(H\toto *)$ over the point
 $*$ and each Lie algebra $\fh$ gives rise to a Lie algebroid $\fh\to *$.
 However, with the above convention the Lie algebroid $\Lf(H)\to *$ is
 \emph{not} isomorphic to the Lie algebroid $\Lf(H\toto {*})$. It rather is
 anti isomorphic. This is an annoying but unavoidable fact if one wants to
 stick to the usual and natural conventions.
\end{warning}

\begin{tabsection}
 We will now line out one main example that will be developed throughout the text
 to illustrate our results.
\end{tabsection}

\begin{example}\label{exmp:gauge_groupoid}
 Let $\pi\from P\to M$ be a principal $H$-bundle. Then the gauge groupoid
 $\frac{P\times P}{H}\toto M$ is defined as follows. The manifold of objects is
 $M$ and the manifold of arrows is the quotient of $P\times P$ by the diagonal
 action of $H$. We denote by $\langle p,q\rangle$ the equivalence class of
 $(p,q)$ in $(P\times P)/H$.
 
 For later reference we recall the construction of charts for
 $\frac{P\times P}{H}$. In order to obtain manifold charts for $(P\times P)/H$,
 let $(U_{i})_{i\in I}$ be an open cover of $M$ such that there exist smooth
 local sections $\sigma_{i}\from U_{i}\to P$ of $\pi$. This yields an atlas
 $(U_i,\kappa_i)_{i \in I}$ of local trivialisations of the bundle
 $\pi \colon P \rightarrow M$ which are given by
 \begin{equation*}
  \kappa_i \colon \pi^{-1} (U_i) \rightarrow U_i \times H,\quad p \mapsto (\pi (p), \delta (\sigma_{i}(\pi(p)), p))
 \end{equation*}
 with $\delta \colon P \times_\pi P \rightarrow H$,
 $(p,q) \mapsto p^{-1}\cdot q$. Here we use $p^{-1}\cdot q$ as the suggestive
 notation for the element in $h\in H$ that satisfies $p.h=q$ (whereas $p^{-1}$
 alone has in general no meaning).
 
 The local trivialisations commute with the right $H$-action on $P$ since
 \begin{equation*}
  \kappa_i (p.h) = (\pi(p.h), \delta (\sigma_{i}(\pi(p.h)),p.h)=(\pi(p), \delta (\sigma_{i}(\pi(p)),p)\cdot h.
 \end{equation*}
 In particular, the trivialisations descent to manifold charts for the arrow
 manifold of the gauge groupoid:
 \begin{equation*}
  K_{ij} \colon \frac{\pi^{-1}(U_i) \times \pi^{-1}(U_j)}{H} \rightarrow U_i
  \times U_j \times H ,\quad \langle p_1,p_2\rangle \mapsto (\pi (p_1), \pi (p_2) ,
  \delta( \sigma_{i}(\pi( p_{1})),p_{1}) \delta (\sigma_{j}(\pi(p_{2})),p_{2})^{-1}).
 \end{equation*}
 One then easily checks that $\alpha( \langle p,q\rangle)\coloneq\pi(p)$,
 $\beta(\langle p,q\rangle)\coloneq \pi(q)$,
 $1(x)=\langle \sigma_{i}(x),\sigma_{i}(x\rangle)$ if $x\in U_{i}$ and
 \begin{equation*}
  m(\langle p,q\rangle ,\langle v,w\rangle)\coloneq\langle p,w\cdot \delta(v,q)\rangle
 \end{equation*}
 complete the definition of a Lie groupoid.
 
 The associated Lie algebroid is naturally isomorphic to the Atiyah algebroid
 $TP/H \to M$. If we identify sections of $TP/H\to M$ with right-invariant
 vector fields on $P$, then the bracket is the usual bracket of vector fields
 on $P$ and the anchor is induced by $T \pi$. To see that this is naturally
 isomorphic to $\Lf(\frac{P\times P}{H}\toto M)$ it is a little more convenient to
 identify $M$ with the submanifold $\frac {\Delta P}{H}$ via the diffeomorphism
 $\frac {\Delta P}{H}\cong \frac {P}{H}\cong M$. Here, $\Delta P$ denotes the
 diagonal $\Delta P\se P\times P$. Then we have
 $\alpha(\langle p,q\rangle)=\langle p,p\rangle$,
 $\beta(\langle p,q \rangle)=\langle q,q\rangle$. Consequently,
 $T^{\alpha}\frac{P\times P}{H}=\frac{0_{P}\times TP}{H}$ and thus we have the
 natural isomorphism
 \begin{equation*}
  \Lf\left(\frac{P\times P}{H}\right)=T^{\alpha}\left.\frac{P\times P}{H}\right|_{\frac{\Delta P}{H}}=\left. \frac{ 0_{P}\times TP }{H}\right|_{\frac{\Delta P}{H}}\cong TP/H
 \end{equation*}
 of vector bundles over $M$. One easily checks that this is in fact a morphism
 of Lie algebroids.
\end{example}

\section{The Lie group structure on the bisections}
\label{sec:the_lie_group_structure_on_the_bisections}

\begin{tabsection}
 It is the task of this section to lift the functor from \eqref{eqn:functor_1}
 to a functor that takes values in locally convex Lie groups. Our main
 technical tool for understanding the Lie group structure on $\Bis(\cG)$ will
 be local additions (cf.\ Definition \ref{def:local_addition}) which respect to
 the fibres of a submersion. This is an adaptation of the construction of
 manifold structures on mapping spaces
 \cite{Wockel13Infinite-dimensional-and-higher-structures-in-differential-geometry,conv1997,michor1980}
 (see also Appendix \ref{Appendix: MFD}). In particular, special cases of our
 constructions are covered by \cite[Chapter 10]{michor1980}, but we aim here at
 a greater generality. In the end, we will have to restrict the functor from
 \eqref{eqn:functor_1} to those Lie groupoids that admit such a local addition.
\end{tabsection}

\begin{definition}(cf.\
 \cite[10.6]{michor1980}) Let $s\from Q\to N$ be a surjective submersion. Then
 a \emph{local addition adapted to $s$} is a local addition
 $\A \from U\opn TQ\to Q$ such that the fibres of $s$ are additively closed
 with respect to $\A$, i.e.\ $\A (v_{q})\in s^{-1}(s(q))$ for all $q\in Q$ and
 $v_{q}\in T_{q}s^{-1}(s(q))$ (note that $s^{-1}(s(q))$ is a submanifold of
 $Q$).
\end{definition}

We will mostly be interested in local addition which respect the source
projection of a Lie groupoid.

\begin{lemma}\label{lemma:
 sourcetarget} If $\cG = (G \toto M)$ is a Lie groupoid with a local addition
 adapted to the source projection $\alpha$, then there exists a local addition
 adapted to the target projection $\beta$.
\end{lemma}

\begin{proof}
 Let $\A \colon U \opn TG \rightarrow G$ be a local addition adapted to
 $\alpha$. Recall that the inversion map $\iota \colon G \rightarrow G$ is a
 diffeomorphism. Hence the tangent $T\iota$ is a diffeomorphism mapping the
 zero-section in $TG$ to itself. In particular, $T\iota (U)$ is an open
 neighbourhood of the zero-section in $TG$. Define
 $\A^{\op{op}} \colon T\iota (U) \opn TG \rightarrow G$ via
 $\A^{\op{op}} = \iota \circ \A \circ T\iota$ and observe that
 $\A^{\op{op}} (0_g) = g$ holds for all $g \in G$. Now $\A$ being a local
 addition implies that
 $(\pi|_{T\iota (U)}, \A^{\op{op}}) \colon T\iota (U) \opn TG \rightarrow G \times G$
 induces a diffeomorphism onto an open neighbourhood of the diagonal, whence
 $\A^{\op{op}}$ is a local addition. To prove that $\A^{\op{op}}$ is
 adapted to $\beta$ we use that $\iota$ intertwines $\alpha$ and $\beta$, i.e.\
 $\beta = \alpha \circ \iota$. Thus $\iota$ maps each submanifold
 $\beta^{-1} (\beta (g))$ to $\alpha^{-1} (\beta (g))$ and thus
 $T\iota (T_g \beta^{-1} (\beta (g))) = T_{g^{-1}} \alpha^{-1} (\beta (g))$. As
 $\A$ is adapted to $\alpha$, we can deduce easily from these facts that
 $\A^{\op{op}}$ is adapted to $\beta$.
\end{proof}

\begin{remark}
 Exchanging the r\^{o}le of $\alpha$ and $\beta$ in the proof of Lemma
 \ref{lemma: sourcetarget}, we see that a local addition adapted to the source
 projection exists if and only if a local addition adapted to the target
 projection exists. 
\end{remark}

\begin{definition}
 We say that that a Lie groupoid $\cG=(G\toto M)$ admits an \emph{adapted local
 addition} if $G$ admits a local addition which is adapted to the source
 projection $\alpha$ (or, equivalently, to the target projection $\beta$). We
 denote the full subcategory of $\cat{LieGroupoids}_{M}$ of Lie groupoids over
 $M$ that admit an adapted local addition by $\cat{LieGroupoids}_{M}^{\A}$.
\end{definition}

\begin{tabsection}
 In the following, all manifolds of smooth mappings are endowed with the smooth
 compact-open topology from \ref{def:smooth_compact_open_topology} and the
 manifold structure from Theorem \ref{thm: MFDMAP}. If a Lie groupoid admits an
 adapted local addition, then we can prove the following result. Note that if
 the manifold $G \times G$ admits tubular neighbourhoods for embedded
 submanifolds, then the assertion of the next lemma is a special case of
 \cite[Proposition 10.8]{michor1980}.
\end{tabsection}

\begin{lemma}\label{lemma:
 composing:submanifold} Suppose $K$ is a compact manifold and
 $\cG = (G \toto M)$ is a Lie groupoid which admits an adapted local addition
 $\A$. Then $C^\infty (K,G \times_{\alpha,\beta} G)$ is a submanifold of
 $C^\infty (K,G \times G)$.
\end{lemma}

\begin{proof}
 Since $\alpha$ is a submersion the fibre product $G \times_{\alpha, \beta} G$
 is a split submanifold of $G \times G$. By Lemma \ref{lemma: sourcetarget} $G$
 admits a local addition $\A$ adapted to $\alpha$ and a local addition
 $\A^{\op{op}}$ adapted to $\beta$. We will identify
 $T(G\times _{\alpha,\beta}G)$ with $TG\times_{T \alpha,T \beta}TG$ as in
 Remark \ref{rem:structure_on_bisections} \ref{rem: tangentgpd}. Then we obtain
 a local addition $\A^{\op{prod}} \coloneq \A \times \A^{\op{op}}$ on
 $T(G \times G)$. Since $\A$ is adapted to $\alpha$ and $\A^{\text{op}}$ is
 adapted to $\beta$, the local addition $\A^{\op{prod}}$ restricts to a local
 addition on the submanifold $G \times_{\alpha, \beta} G$ (i.e.\ the
 submanifold is additively closed).
 
 Now let
 $g \in C^\infty (K, G \times_{\alpha,\beta} G) \subseteq C^\infty (K,G \times G)$.
 We consider the chart $(O_g, \varphi_g)$ for $g$ on $C^\infty (K, G \times G)$
 which is induced by $\A^{\op{prod}}$. As $G \times_{\alpha,\beta} G$ is
 additively closed with respect to $\A^{\op{prod}}$ we derive the condition:
 \begin{displaymath}
  f \in C^\infty (K, G \times_{\alpha, \beta} G) \cap O_g \Leftrightarrow
  \varphi_g (f) = (\pi_{T(G \times G)}, \A^{\op{prod}})^{-1} (g,f) \in \Gamma
  (g^* T(G \times_{\alpha ,\beta}G))
 \end{displaymath}
 Notice that the linear subspace $\Gamma (g^* T(G \times_{\alpha ,\beta}G))$ is
 closed. This follows from the continuity of the point evaluations
 $\op{ev}_x \colon \Gamma (g^* T(G \times G)) \rightarrow T(G\times G), f \mapsto f(x)$
 (cf.\ \cite[Proposition 3.20]{alas2012}): As the vector subspace
 $T_y(G \times_{\alpha ,\beta} G)$ is complemented (thus closed) for all
 $y \in G \times_{\alpha, \beta} G$, we can write
 $\Gamma (g^* T(G \times_{\alpha ,\beta}G))$ as an intersection of closed
 subspaces
 $\bigcap_{x \in K} ((\pi_{TG \times G}^*g) \circ \op{ev}_x)^{-1} (T_{g(x)} G \times_{\alpha ,\beta} G)$.
 Thus the canonical chart restricts to a submanifold chart
 $\varphi_g \colon O_g \cap C^\infty (K , G \times_{\alpha,\beta} G) \rightarrow \Gamma (g^* T(G \times_{\alpha ,\beta}G)) \subseteq \Gamma (g^* T(G \times G))$.
\end{proof}

\begin{tabsection}
 We now use adapted local additions to endow the sections of a submersion with
 a smooth manifold structure.
\end{tabsection}

\begin{proposition}\label{Proposition:
 SectMFD} Let $N$ be a locally convex and locally metrisable manifold and $K$
 be a compact manifold. Furthermore, let $s \colon N \rightarrow K$ be a
 submersion. If there exists a local addition
 $\A \colon U \opn TN \rightarrow N$ adapted to $s$, then the set
 \begin{displaymath}
  \Gamma(K \xleftarrow{s} N) \coloneq \{ \sigma \in C^\infty (K,N) \mid s \circ
  \sigma = \id_K \}
 \end{displaymath}
 is a submanifold of $C^\infty (K,N)$. Furthermore, the model space of an open
 neighbourhood of $\sigma\in \Gamma (K \xleftarrow{s} N)$ is the closed
 subspace
 \begin{equation*}
  E_\sigma \coloneq   \{\gamma\in \Gamma (\sigma^* TN)\mid \forall x \in K, \  \gamma(x)\in T_{\sigma(x)}s^{-1}(x) \}
 \end{equation*}
 of all vertical sections in $\Gamma (\sigma^* TN)$.
\end{proposition}

\begin{proof}
 Endow $C^\infty (K,N)$ with the manifold structure from Theorem \ref{thm:
 MFDMAP}.\ref{thm:manifold_structure_on_smooth_mapping_b} constructed with
 respect to the local addition $\A$ that is adapted to $s$. We claim that for a
 section $\sigma$ of the submersion $s$ the canonical charts
 $(\varphi_\sigma, O_\sigma)$ of $C^\infty (K,N)$ from \ref{thm:
 MFDMAP}.\ref{thm:manifold_structure_on_smooth_mapping_a} define submanifold
 charts for $\Gamma (K \xleftarrow{s} N)$.
 To see this, consider $g \in O_\sigma$ and recall
 $\varphi_\sigma (g) = (\pi_{TN}, \A)^{-1} \circ (\sigma , g)$. Since the local
 addition $\A$ is adapted to $s$, the formula for $\varphi_\sigma$ shows that
 \begin{equation*}
  g \in \Gamma (K \xleftarrow{s} N) \cap O_\sigma \Leftrightarrow
  \varphi_\sigma (g) \in E_{\sigma}\cap C^{\infty}(K, U),
 \end{equation*}
 where $U\opn TN$ is as is Theorem \ref{thm: MFDMAP}. For $x \in K$ we define
 the evaluation map
 $\op{ev}_x \colon \Gamma (\sigma^* TN) \rightarrow TN, f \mapsto f(x)$. It
 is easy to see that the evaluation maps are continuous, since they are
 continuous in each chart (cf.\ \cite[Proposition 3.20]{alas2012}). The vector
 subspace $E_\sigma \subseteq \Gamma (\sigma^* TN)$ is thus closed as an
 intersection of closed subspaces
 $E_\sigma = \bigcap_{x \in K} ((\pi_{TN}^*\sigma) \circ \op{ev}_x)^{-1} (T_{f(x)} s^{-1} (x))$.
 In particular, $\varphi_{\sigma}$ is a submanifold chart and the assertion
 follows.
\end{proof}

\begin{remark}
 In certain cases the submanifold $\Gamma (K \xleftarrow{s} N)$ constructed in
 Proposition \ref{Proposition: SectMFD} will be a split submanifold of
 $C^\infty (K,N)$. For example this will happen if $N$ is a finite dimensional
 manifold (see \cite[Proposition 10.10]{michor1980}). The same proof carries
 over to the following slightly more general setting: If $N$ is a manifold such
 that for each embedded submanifold $Y \subseteq N$ there exists a tubular
 neighbourhood, then $\Gamma (K \xleftarrow{s} N)$ is a split submanifold of
 $C^\infty (K,N)$.
\end{remark}

Using this manifold structure we can finally prove the first main result of this
paper.

\begin{Theorem}\label{theorem:
 A} Suppose $M$ is compact and $\cG=(G  \toto M)$ is a locally convex and
 locally metrisable Lie groupoid over $M$ which admits an adapted local
 addition. Then $\Bis (\cG)$ is a submanifold of $C^{\infty}(M,G)$ (with the
 manifold structure from Theorem \ref{thm: MFDMAP}). Moreover, the induced
 manifold structure and the group multiplication
 \begin{equation*}
  (\sigma \star \tau ) (x) \coloneq \sigma ((\beta \circ \tau)(x))\tau(x)\text{ for }  x \in M
 \end{equation*}
 turn $\Bis(\cG)$ into a Lie group modelled on
 \begin{equation*}
  E_1 \coloneq   \{\gamma\in C^{\infty}(M,TG) \mid \forall x \in M, \  \gamma(x)\in T_{1_{x}}s^{-1}(x) \}.
 \end{equation*}
\end{Theorem}

\begin{tabsection}
 Note that $E_{1}$ is in fact isomorphic to the space of sections of the Lie
 algebroid $\Lf(\cG)$ associated to $\cG$. It will be the content of Section
 \ref{sec:the_lie_algebra_of_the_group_of_bisections} to analyse this
 isomorphism and show that it is natural and respects the Lie bracket (up to a
 sign).
\end{tabsection}

\begin{proof}[of Theorem \ref{theorem: A}] 
In Proposition \ref{Proposition: SectMFD} we endowed
 the space of sections $\Gamma (M \xleftarrow{\alpha} G)$ with the structure of
 a submanifold of $C^\infty (M,G)$. Observe that by Theorem \ref{thm: MFDMAP}
 \ref{thm:manifold_structure_on_smooth_mapping_h} the map
 \begin{displaymath}
  C^\infty (M,G) \rightarrow C^{\infty} (M,M),\quad f \mapsto \beta \circ f
 \end{displaymath}
 is smooth, and so is its restriction $\beta_*$ to
 $\Gamma (M \xleftarrow{\alpha}   G)$. Recall from \cite[Corollary
 5.7]{michor1980}
 that the subset of all diffeomorphisms $\Diff (M) \subseteq C^\infty (M,M)$ is
 open. By construction $\Bis (\cG) = (\beta_*)^{-1} (\Diff (M))$ is thus an
 open submanifold of $\Gamma (M \xleftarrow{\alpha} G)$, and thus also a
 submanifold of $C^{\infty}(M,G)$.
 
 As $\Gamma (M \xleftarrow{\alpha} G)$ is locally metrisable by Proposition
 \ref{Proposition: SectMFD}, so is $\Bis (\cG)$. Thus all we have to show is
 that the group operations \eqref{eq: BISGP1} and \eqref{eq: BISGP2} of the
 group of bisections are smooth with respect to the submanifold structure. We
 begin with the group multiplication.
 
 By Theorem \ref{thm: MFDMAP}
 \ref{thm:manifold_structure_on_smooth_mapping_h} the maps
 $\beta_* \colon \Bis (\cG) \rightarrow \Diff (M)$ and
 $m_* \colon C^{\infty} (M, G \times_{\alpha, \beta} G) \rightarrow C^\infty (M,G)$
 are smooth. Furthermore, by Theorem \ref{thm: MFDMAP}
 \ref{thm:manifold_structure_on_smooth_mapping_e} the composition
 $\op{Comp} \colon C^{\infty} (M,G) \times C^{\infty} (M,M) \rightarrow C^{\infty} (M,G)$
 is smooth. Hence for $\sigma, \tau \in \Bis (\cG)$ the map
 \begin{displaymath}
  \mu \colon \Bis (\cG)^2 \rightarrow C^\infty (M, G)^2, (\sigma , \tau)
  \mapsto (\sigma \circ \beta \circ \tau ,\tau) = (\op{Comp} (\sigma , \beta_* (\tau)),
  \tau)
 \end{displaymath}
 is smooth. Let $\Delta \colon M \rightarrow M \times M$ be the diagonal map.
 The canonical identification
 \begin{displaymath}
  h \colon C^\infty (M, G)^2 \rightarrow C^\infty (M,G \times G) , \quad (f , g)
  \mapsto (f, g) \circ \Delta.
 \end{displaymath}
 is a diffeomorphism by an argument analogous to \cite[Proposition
 10.5]{michor1980}. Indeed, the proof carries over verbatim to our setting of
 infinite dimensional locally convex manifolds since $G$ admits a local
 addition. Observe that the map $h\circ \mu$ takes its image in the submanifold
 $C^\infty (M, G \times_{\alpha,\beta} G)$ (cf.\ Lemma \ref{lemma:
 composing:submanifold}). Hence we can rewrite the multiplication formula
 \eqref{eq: BISGP1} for $\sigma, \tau \in \Bis (\cG)$ as a composition of
 smooth maps:
 \begin{displaymath}
  \sigma \star \tau = (\sigma \circ \beta \circ \tau) \cdot \tau = m_* \circ h
  \circ \mu (\sigma , \tau).
 \end{displaymath}
 In conclusion the group multiplication is smooth with respect to the manifold
 structure on $\Bis (\cG)$. 
 
 We are left to prove that inversion in
 $\Bis (\cG)$ is smooth. To this end let us recall the formula \eqref{eq:
 BISGP2} for the inverse of $\sigma \in \Bis (\cG)$:
 \begin{displaymath}
   \sigma^{-1} = (\sigma \circ (\beta \circ \sigma)^{-1})^{-1} = \iota_* \op{Comp}(\sigma , (\beta_* (\sigma))^{-1})
 \end{displaymath}
 Here $\iota \colon G \rightarrow G$ is the inversion in the groupoid, whence
 $\iota_* \colon C^{\infty} (M , G ) \rightarrow C^\infty (M,G)$ is smooth by
 Theorem \ref{thm: MFDMAP} \ref{thm:manifold_structure_on_smooth_mapping_h}.
 Furthermore, $\beta_*$ maps $\Bis (\cG)$ into the open submanifold
 $\Diff (M)$. Inversion of $\beta_* (\sigma)$ in \eqref{eq: BISGP2} is thus
 inversion in the group $\Diff (M)$. The group $\Diff (M)$ is a Lie group with
 respect to the open submanifold structure induced by $C^\infty (M,M)$ (see
 \cite[Theorem 11.11]{michor1980}). We conclude from \eqref{eq: BISGP2} that
 inversion in the group $\Bis (\cG)$ is smooth and thus $\Bis (\cG)$ is a Lie
 group.
\end{proof}

\begin{proposition}
 Suppose $\cG=(G \toto M)$ and $\cH=(H \toto M)$ are Lie groupoids over the
 compact manifold $M$ and that $\cG$ and $\cH$ admit an adapted local addition.
 If then $f\from \cG\to \cH$ is a morphism of Lie groupoids over $M$, then
 \begin{equation*}
  \Bis(f)\from \Bis(\cG)\to \Bis(\cH),\quad \sigma \mapsto f \circ \sigma
 \end{equation*}
 is a smooth morphism of Lie groups.
\end{proposition}

\begin{proof}
 The map $f_{*}\from C^{\infty}(M,G)\to C^{\infty}(M,H)$,
 $\gamma\mapsto f \circ \gamma$ is smooth by Theorem \ref{thm: MFDMAP}
 \ref{thm:manifold_structure_on_smooth_mapping_h}. Thus its restriction to the
 submanifold $\Gamma(M \xleftarrow{\alpha} G)$ and in there to the open subset
 $\Bis(\cG)$ is smooth. That it is a group homomorphism follows directly from
 the definition.
\end{proof}

\begin{remark}\label{rem:functorial_interpretation}
 Suppose $M$ is a compact manifold. Then we consider the full subcategory
 $\cat{LieGroupoids}_{M}^{\A}$ of $\cat{LieGroupoids}_{M}$ whose objects
 are locally convex and locally metrisable Lie groupoids with object space $M$
 that admit an adapted local addition. Then the results of this section show
 that $\Bis$ may be regarded as a functor
 \begin{equation*}
  \Bis\from \cat{LieGroupoids}_{M}^{\A}\to \cat{LieGroups},
 \end{equation*}
 where $\cat{LieGroups}$ denotes the category of locally convex Lie groups.
\end{remark}

\begin{proposition}
 Under the assumptions from Theorem \ref{theorem: A}, the natural action
 \begin{equation*}
  \gamma \colon \Bis (\cG) \times G \rightarrow G ,\quad (\psi , g) \mapsto
  \psi (\beta (g)) g.
 \end{equation*}
 is smooth, as well as the restriction of this action to the $\alpha$-fibre
 \begin{equation*}
  \gamma_{m} \colon \Bis (\cG) \times \alpha^{-1}(m) \rightarrow \alpha^{-1}(m) ,\quad (\psi , g) \mapsto
  \psi (\beta (g)) g
 \end{equation*}
 for each $m\in M$.
\end{proposition}

\begin{proof}
 The action $\gamma$ is given as the composition
 $\gamma (\psi , g) = m (\ev (\psi , \beta (g)) , g)$ for
 $(\psi , g) \in \Bis (\cG) \times G$. Here
 $\ev \colon \Bis (\cG) \times M \rightarrow G$ is the canonical evaluation
 map, which is smooth by \ref{thm: MFDMAP}
 \ref{thm:manifold_structure_on_smooth_mapping_d}. Thus the action $\gamma$ is
 smooth as a composition of smooth maps.
 
 From $\alpha(\psi (\beta (g)) g)=\alpha(g)$ it follows that the action
 $\gamma$ preserves the $\alpha$-fibres. Since $\alpha$ is a submersion,
 $\alpha^{-1}(m)$ is a submanifold. Consequently, the action restricted to
 $\alpha^{-1}(m)$ is also smooth.
\end{proof}

\begin{tabsection}
 We will now give examples for Lie groupoids which admit an adapted local
 addition. Hence the following classes of Lie groupoids we can apply the
 previous results of this section to.
\end{tabsection}

\begin{proposition}\label{prop: ex:locadd}
 Suppose $\cG = (G \toto M)$ is a Lie groupoid such that $G,M$ are Banach
 manifolds and $M$ admits smooth partitions of unity. Then $\cG$ admits an
 adapted local addition.
\end{proposition}

\begin{tabsection}
 The statement suggest similarities to \cite[Proposition
 3.2]{Rybicki02A-Lie-group-structure-on-strict-groups}. However, the difference
 between \cite[Proposition
 3.2]{Rybicki02A-Lie-group-structure-on-strict-groups} and the previous
 proposition is that we will construct an adapted local addition on all of $G$,
 whereas in \cite[Proposition
 3.2]{Rybicki02A-Lie-group-structure-on-strict-groups} a local addition is only
 constructed on a neighbourhood of $M$ in $G$. This allows us to view
 $\Bis(\cG)$ as a \emph{submanifold} of $C^{\infty}(M,G)$ and to use the known
 results on mapping spaces, rather than constructing an auxiliary manifold
 structure on $\Bis(\cG)$.
\end{tabsection}

\begin{proof}[of
 Proposition \ref{prop: ex:locadd}] We first recall some concepts from \cite[\S
 IV.3]{Lang95Differential-and-Riemannian-manifolds}. Suppose $X$ is an
 arbitrary Banach manifold and write $\pi_{TX}\from TX\to X$ for the projection
 of the tangent bundle. Then a vector field $TX\to T(TX)$ on $TX$ is said to be
 of \emph{second order} if
 \begin{equation}\label{eqn:second_order}
  T(\pi_{TX}) (F(v))=v
 \end{equation}
 holds for all $v\in TX$. For each $s\in \R$ we denote by
 $s_{TX}\from TX\to TX$ the vector bundle morphism which is given in each fibre
 by multiplication with $s$. With this notion fixed we define a second order
 vector field $F\from TX\to T(TX)$ to be a \emph{spray} if
 \begin{equation}\label{eqn:spray}
  F(s\cdot v)=T( s_{TX})(s\cdot F(v))
 \end{equation}
 holds for all $s\in\R$ and all $v\in TX$. For each $v\in TX$ there exists a
 unique integral curve $\beta_{v}$ of $F$ with initial condition $v$. The
 subset $W\se TX$ such that $\beta_{v}$ is defined on $[0,1]$ is an open
 neighbourhood of the zero section in $TX$. Then the \emph{exponential map} of
 $F$ is defined to be
 \begin{equation*}
  \exp_{F}\from W\to X,\quad v\mapsto \pi_{TX}( \beta_{v}(1)).
 \end{equation*}
 The restriction of $\exp_{F}$ to $T_{x}X$ is for each $x\in X$ a local
 diffeomorphism at $0_{x}$. By the Inverse Function Theorem there exists
 $U_{x}\opn TX\cap W$ with $0_{x}\in U_{x}$ such that
 $\left.(\pi\times \exp_{F})\right|_{U_{x}}$ is a diffeomorphism onto its
 image. Consequently, we obtain a local addition
 \begin{displaymath}
  \A_{F}\from U\coloneq\cup_{x\in X}U_{x}\to X,\quad v\mapsto \exp_{F}(v).
 \end{displaymath}
 
 Thus sprays are the key to constructing local additions on Banach manifolds.
 We now assume that this manifold is the manifold of arrows of the Lie groupoid
 $(G\toto M)$. For the rest of this proof we identify $M$ with the submanifold
 $1(M)$ of $G$ and $TM$ with the submanifold $T1(TM)$ of $TG$. There exists a
 spray $F\from TM\to T(TM)$ on $M$ by \cite[Theorem
 IV.3.1]{Lang95Differential-and-Riemannian-manifolds}. With respect to the
 mentioned identifications, we may interpret $F$ as a section of
 $\left.T^{T \alpha}TG\right|_{TM}\to TM$ that satisfies
 \eqref{eqn:second_order} and \eqref{eqn:spray} for all $v\in TM$ and all
 $s\in \R$.
 
 We now want to extend this spray by right translation. To this end, recall
 from Remark \ref{rem:structure_on_bisections} \ref{rem: tangentgpd} that
 $T\cG=(TG\toto TM)$ is also a Lie groupoid, where we take the tangent maps at
 all levels.\bigskip
 
 \textbf{Claim: The vertical right-invariant extension
 $\overrightarrow{F}\from TG\to T^{T \alpha}TG$ of $F$ is a spray satisfying
 $\overrightarrow{F}(v)\in T^{T \alpha}TG$ for all $v\in TG$.}
 
 We first establish \eqref{eqn:second_order} for all $v\in TG$. To this end we
 compute
 \begin{align*}
  T(\pi_{TG})(\overrightarrow{F}(v))&=T(\pi_{TG})(TR_{v}(F(T \beta(v))))=
  T(\pi_{TG}\circ Tm)(F(T \beta(v)),0_{v})\\
  &=T(m \circ \pi_{TG\times TG})(F(T \beta(v)),0_{v})=Tm(T \pi_{TG}(F(T \beta(v))),v)=Tm(T \beta(v),v)=v,
 \end{align*}
 where we have used $T \pi_{TG}(F(T \beta(v)))=T \beta(v)$ since
 $T \beta(v)\in TM$, that $T \beta(v)$ is an identity in $(TG\toto TM)$ and
 that
 \begin{equation*}
  \xymatrix{TG\times_{T \alpha,T \beta}TG \ar[r]^(.65){Tm} \ar[d]_{\pi_{TG\times TG}} & TG \ar[d]^{\pi_{TG}}\\ 
  G\times_{\alpha,\beta}\ar[r]^{m}G & G
  }
 \end{equation*}
 commutes. That \eqref{eqn:spray} holds for all $s\in \R$ and all $v\in TG$
 follows from the linearity of the tangent maps on each tangent space and the
 equality $R_{sv}\circ s_{TG} =  s_{TG} \circ R_{v}$, which imply
 \begin{align*}
  \overrightarrow{F}(s\cdot v)&=TR_{sv}(F(T \beta(s\cdot v)))=TR_{sv}(F(s\cdot T \beta(v)))=
  TR_{sv}(Ts_{TG}(s\cdot F(T \beta(v))))\\
  &=Ts_{TG}(TR_{v}(s\cdot F(T \beta(v))))=Ts_{TG}(s\cdot TR_{v}(F(T \beta(v))))=
  Ts_{TG}(s\cdot \overrightarrow{F}(v)).
 \end{align*}
 
 \textbf{Claim: The local addition $\A_{\overrightarrow{F}}$ constructed from
 $\overrightarrow{F}$ is adapted to the source projection $\alpha$.}
 
 Suppose $\beta_{v}\from [0,1]\to TG$ is an integral curve for
 $\overrightarrow{F}$. By definition we have
 \begin{equation*}
  \alpha(\A_{\overrightarrow{F}}(v))=\alpha(\pi_{TG}(\beta_{v}(1)))=\pi_{TM}(T \alpha(\beta_{v}(1))).
 \end{equation*}
 Thus it suffices by \cite[Proposition
 2.11]{Lang95Differential-and-Riemannian-manifolds} to check that
 $T \alpha \circ \beta_{v}\from [0,1]\to TM$ is an integral curve for the zero
 vector field. The latter is in fact the case since we have
 \begin{equation*}
  \left(T \alpha \circ \beta _{v} \right)'(t)=T\left(T \alpha (\beta_{v}'(t)) \right)
  =T\left( T \alpha (\underbrace{\overrightarrow{F}(\beta_{v}(t))}_{\in \ker T_{\beta_{v}(t)}T \alpha}) \right)=0.
 \end{equation*}
 
 Putting these two proven claims together establishes the proof of the
 proposition.
\end{proof}

\begin{remark}
 Let $\cG=(G\toto M)$ be a locally trivial Lie groupoid, i.e., one for which
 $(\beta,\alpha)\from G \to M\times M$ is a surjective submersion. Then $\cG$
 is equivalent over $M$ to the gauge groupoid of a principal bundle (the
 argument from \cite[\S
 1.3]{Mackenzie05General-theory-of-Lie-groupoids-and-Lie-algebroids} carries
 verbatim over to our more general setting). Thus the following proposition
 implies that each locally trivial Lie groupoid with locally exponential vertex
 group and finite-dimensional space of objects admits a local addition.
\end{remark}

\begin{proposition}
 Let $\pi\from P\to M$ be a principal $H$-bundle with finite-dimensional base
 $M$ and locally exponential structure group $H$. Then the associated gauge
 groupoid $\frac{P\times P}{H}\toto M$ admits an adapted local addition.
\end{proposition}

\begin{proof}
 We begin with the construction of a local addition for the Lie group $H$. As
 $H$ is locally exponential, the exponential map
 $\exp_H \colon L(H) \rightarrow H$ restricts to a diffeomorphism on a
 zero-neighbourhood. Fix a convex zero-neighbourhood $W \opn L(H) =T_e H$ (the
 tangent space at the identity $e \in H$) together with an
 identity-neighbourhood $V \opn H$ such that
 $\psi \coloneq (\exp_H|_W^V)^{-1} \colon V \rightarrow W$ becomes a manifold
 chart for $H$. By construction this chart satisfies the following properties:
 \begin{enumerate}
  \item $\psi (e) = 0$
  \item If $x \in V$ and for $k \in H$ the product $kxk^{-1}$ is also contained
        in $V$, then $\psi (kxk^{-1}) = \op{Ad} (k) (\psi (x))$ holds. Here
        $\op{Ad} \colon H \rightarrow \op{Aut} (L(H))$ is the adjoint
        representation of $H$ on its Lie-algebra.
 \end{enumerate}
 Let $m_H$ be the group multiplication in $H$ and
 $\lambda_h = m_H (h, \mathinner\cdot)$, $\rho_{h}=m_H(\mathinner\cdot,h)$ for
 $h \in H$. The tangent bundle $TH$ of the Lie group $H$ admits the following
 trivialisation
 $\Phi \colon H \times T_e H \rightarrow TH,  (h, V) \mapsto h.V = 0_h \cdot V = T \lambda_h (V)$.
 Hence $\tilde{W} \coloneq \Phi (H\times W)$ is an open neighbourhood of the
 zero-section in $TH$. We define a smooth map
 \begin{displaymath}
  \A_H \colon \tilde{W} \rightarrow H ,\quad \A_H(h.V) \coloneq h\cdot
  \exp_{H}(V),
 \end{displaymath}
 which obviously satisfies $\A_H (0_h) = h$ for all $h \in H$. The inverse to
 $(h.V)\mapsto (h,\A_{H}(h.V))$ is given by $(h,h')\mapsto h.\psi( h^{-1}h')$
 and thus $\A_{H}$ is a local addition. Moreover, the local addition $\A_H$ is
 left-invariant and right-invariant, i.e.\ for $h\in H$, $V_{1}\in \tilde{W}$
 and $V_{2}\in T \rho_{h^{-1}}(\tilde{W})$ we have
 \begin{equation}\label{eqn:equiv_loc_add_1}
  \A_H (h.V_{1})= h\cdot\A_H(V_{1})
 \end{equation}
 by definition and, if $V_{2}\in T_{h'}H$,
 \begin{equation}\label{eqn:equiv_loc_add_2}\begin{aligned}
  \A_{H}(V_{2}.h)	&= \A_{H}((h'h(h'h)^{-1}).V_{2}.h)=h'h\cdot \exp_{H}(\Ad(h^{-1})(h'^{-1}.V_{2}))= h'hh^{-1}\exp_{H}(h'^{-1}.V_{2})\cdot h\\ 
  &=\A_{H}(V_{2})\cdot h.
  \end{aligned}
 \end{equation}
 We use the local addition on $H$ to construct the desired local addition on
 the gauge groupoid $\frac{P \times P}{H} \toto M$.
 
 We will use the notation introduced in Example \ref{exmp:gauge_groupoid} for
 the gauge groupoid. To simplify the notation, define
 $\gamma_{p_{1},p_{2}} \coloneq\delta( \sigma_{i}(\pi( p_{1}),p_{1})) \delta (\sigma_{i}(\pi(p_{2})),p_{2})^{-1}$,
 set $U_{ij} \coloneq U_i \cap U_j$ and denote by
 $k_{ji} \colon U_{ij} \rightarrow H$ the smooth map
 $x\mapsto\delta(\sigma_{j}(x),\sigma_{i}(x))$.\bigskip
 
 \textbf{Construction of the local addition in charts:} Fix $i \in I$ and let
 $\A_M \colon TM \rightarrow M$ be a local addition for $M$. Since $M$ is
 finite dimensional a globally defined local addition always exists by \cite[p.
 441]{conv1997}. For $i \in I$ we set $W_i \coloneq \A_M^{-1} (U_i) \cap TU_i$
 and $V_{i}:=\pi_{TU_{i}}(W_{i})$. Then $V_{i}$ is open in $U_{i}$ and $W_i$ is
 an open neighbourhood of the zero-section in $TV_i$. By making the indexing
 set $I$ larger and shrinking $U_{i}$ if necessary we may assume that
 $(V_{i})_{i\in I}$ is still an open cover of $M$.
 
 We now now want to use similar trivialisations as in Example
 \ref{exmp:gauge_groupoid} for the gauge groupoid of the principal $TH$-bundle
 $T \pi\from TP\to TM$. In order to obtain trivialisations with a slightly more
 specialised property, we proceed as follows. Since $V_{i}\se U_{i}$, for the
 open cover $(TV_{i})_{i\in I}$ there exist the local sections
 $T \sigma_{i}\from TV_{i}\to TP$ of $T \pi$. Now choose a connection on $TP$,
 which we interpret as a decomposition $TP=T^{v}P\oplus \mf{H}$ of the tangent
 bundle of $P$ into the vertical bundle $T^{v}P:= \ker(T \pi)$ and an
 $H$-equivariant horizontal complement. The existence of such a connection is
 ensured by the smooth paracompactness of the base $TM$ by constructing a
 connection on $\left.TP\right|_{TV_{i}}$ and patching them together with a
 partition of unity. From this we obtain the projection
 $\pi^{\mf{H}}\from TP\to \mf{H}$, which is a morphism of vector bundles over
 $TM$. Consequently,
 $\wt{\sigma}_{i}\coloneq  \pi^{\mf{H}}\circ T \sigma_{i}\from TV_{i}\to TP$ is
 another system of sections of $T \pi$. From this we deduce for
 $v_x \in TV_i \cap TV_j$ the formula
 \begin{equation*}
  \wt{\sigma}_{i}(v_{x})=\wt{\sigma}_{j}(v_{x})\cdot k_{ji}(x),
 \end{equation*}
 since $\wt{\sigma}_{i}(v_{x})$ and $\wt{\sigma}_{j}(v_{x})$ are both
 \emph{horizontal} tangent vectors in $T_{\sigma_{i}(x)}P$ and
 $T_{\sigma_{j}(x)}P$ respectively. If we denote as in Example
 \ref{exmp:gauge_groupoid} the trivialisations associated to the sections
 $\wt{\sigma}_{i}$ by
 \begin{equation*}
  \wt{TK_{ij}}\from \frac{T \pi^{-1}(TV_{i})\times T \pi^{-1}(TV_{j})}{TH}\to TV_{i}\times TV_{j}\times TH,
 \end{equation*}
 then the associated chart changes are given by
 \begin{equation*}
  \wt{TK_{ij}}\circ \wt{TK_{mn}}^{-1}\from TV_{i}\times TV_{j}\times TH\to TV_{m}\times TV_{n}\times TH,\quad
  (V_x,V_y,V_{h})\mapsto (V_x,V_y,k_{mi}(x)\cdot V_{h}\cdot k_{nj}(y)^{-1}).
 \end{equation*}
 Here the product in the third component is the product in the tangent group
 $TH$. Now we can define a smooth map
 \begin{equation}\label{eq: GAU:local:Add}
  \A_{ij} \colon \wt{TK}_{ij}^{-1} (V_i \times V_j \times \tilde{W}) \rightarrow K_{ij}( U_i \times U_j \times H) ,\quad \A_{ij} \coloneq K_{ij}^{-1} \circ (\A_M \times \A_M \times \A_H ) \circ  \wt{TK}_{ij}.
 \end{equation}
 By construction
 $\wt{TK}_{ij} (0_{\langle p,q\rangle}) = (0_{\pi (p)}, 0_{\pi (q)} , 0_{\gamma_{u_2,u_1}})$
 holds for $\langle p,q\rangle $ in the domain of $K_{ij}$. Since $\A_M$ and
 $\A_H$ are local additions, for all such $\langle p,q\rangle$ we obtain
 $\A_{ij} (0_{\langle p,q\rangle})= \langle p,q\rangle$.\bigskip
 
 \textbf{Claim: For $i,j,m,n \in I$ the maps $\A_{ij}$ and $\A_{mn}$ coincide
 on the intersection of their domains.} Assume that the intersection
 $\frac{P_{im}\times P_{jn}}{H}$ of the domains of $K_{ij}$ and $K_{mn}$ is
 non-empty. Let $\langle p,q\rangle$ be an element of
 $\frac{P_{im}\times P_{jn}}{H}$ with $x \coloneq \pi (p)$ and
 $y \coloneq \pi(q)$. We will show that for
 $V_{\langle p,q\rangle} \in T_{\langle p,q\rangle} \frac{P \times P}{H} \cap \op{dom} \A_{ij} \cap \,\op{dom} \A_{mn}$
 the mappings $\A_{ij}$ and $\A_{mn}$ yield the same image. The image
 $\wt{TK}_{ij} (V_{\langle p,q\rangle}) = (V_x,V_y,V_{\gamma_{p,q}}) \in T_x V_{im} \times T_y V_{jn} \times T_{\gamma_{p,q}} H$
 is related to the image under $\wt{TK}_{mn}$ via the formula
 \begin{equation}\label{eq: chch:relation}
  \wt{TK}_{mn} (V_{\langle p,q\rangle}) = (V_x, V_y, k_{mi} (x) \cdot V_{\gamma_{p,q}} \cdot k_{nj} (y)^{-1})
 \end{equation}
 The change of charts formula \eqref{eq: chch:relation} shows that the first
 two components of the image are invariant under change of charts. By
 definition of the maps $\A_{ij}$ and $\A_{mn}$, the two maps coincide in these
 components since they are just a restriction of the map $\A_M$. We compute now
 a formula for the third component of
 $K_{mn} \circ \A_{mn} (V_{\langle p,q\rangle})$, which is given by \eqref{eq:
 chch:relation}, \eqref{eq: GAU:local:Add}, \eqref{eqn:equiv_loc_add_2} and
 \eqref{eqn:equiv_loc_add_1} by
 \begin{equation*}
  \A_H (k_{mi} (x) \cdot V_{\gamma_{p,q}} \cdot k_{nj} (y)^{-1}) = k_{mi} (x) \A_H (V_{\gamma_{p,q}}) k_{nj} (y)^{-1}.
 \end{equation*}
 We can thus conclude
 $\A_{ij} (V_{\langle p,q\rangle}) = \A_{mn}  (V_{\langle p,q\rangle})$. As the
 smooth maps $\A_{ij}, (i,j) \in I^{2}$ coincide on the intersection of their
 domains, we obtain a well-defined smooth map
 \begin{displaymath}
  \A_{\Gau} \colon \bigcup_{(i,j) \in I^{2}} \op{dom} \A_{ij} \opn
  T\left(\frac{P\times P}{H}\right) \rightarrow G ,\quad V \mapsto \A_{ij} (V)
  \text{ if } V \in \op{dom} \A_{ij}.
 \end{displaymath}
 
 \textbf{Claim: $\A_{\Gau}$ is a local addition adapted to the source
 projection:} By construction $\A_{\Gau}$ is defined on an open
 neighbourhood of the zero-section. The local additions $\A_M$ and $\A_H$ do
 not depend on the chart $K_{ij}$. Using this fact, an easy computation in
 charts shows that
 $(\pi_{T\frac{P\times P}{H}}|_{\bigcup_{(i,j) \in I^{2}} \op{dom } \A_{ij}}, \A_{\Gau})$
 is a diffeomorphism onto an open neighbourhood of the diagonal in
 $\frac{P\times P}{H} \times \frac{P\times P}{H}$. We conclude that
 $\A_{\Gau}$ is a local addition.
 
 Recall that the source projection of
 the gauge groupoid $\frac{P\times P}{H}$ is the mapping
 $\alpha \colon \frac{P\times P}{H} \rightarrow M, \langle p,q\rangle \mapsto \pi (q)$.
 Computing in the chart $K_{ij}$ we see that $\alpha$ is just the projection
 onto the second factor of the product. Hence the kernel of $T\alpha$ at a
 point $\langle p,q\rangle$ in this chart is the subspace of elements
 $V_{\langle p,q\rangle} \in T_{\langle p,q\rangle} \frac{P\times P}{H}$ with
 $\pr_2 \circ \wt{TK}_{ij} (V_{\langle p,q\rangle}) = 0_{\pi (q)}$. The local
 addition $\A_M$ satisfies $\A_M (0_{\pi (q)}) = \pi (q)$. By construction,
 this implies
 $\A_{\Gau} (\ker T_{\langle p,q\rangle} T\alpha ) \subseteq \pi^{-1} (\pi(q))$.
 Summing up, the local addition $\A_{\Gau}$ is adapted to the source
 projection.
\end{proof}

\begin{remark}
 The same argument as in the previous proof shows that under the same
 assumptions the associated Lie group bundle $\frac{P\times H}{H}\toto M$ has an
 adapted local addition. In fact, $\frac{P\times H}{H}\toto M$ may be considered as
 the arrow manifold of the subgroupoid $\frac{P\times_{M} P}{H}\toto M$ (with
 equal source and target projection) of $\frac{P\times P}{H}\toto M$, and the
 local addition constructed in the previous proof restricts to a local addition
 on this submanifold with the desired properties.
\end{remark}

\begin{example}\label{exmp:gauge_and_automorphism_group}
 For a principal $H$-bundle $\pi\from P\to M$ with locally exponential
 structure group $H$ and compact $M$ we thus obtain a Lie group structure on
 $\Bis(\frac{P\times P}{H})$. Moreover, the natural map
 \begin{equation*}
  \beta_{*}\from  \Bis\left(\frac{P\times P}{H}\right)\to \Diff(M)
 \end{equation*}
 is smooth by Theorem \ref{thm: MFDMAP}
 \ref{thm:manifold_structure_on_smooth_mapping_h}. Its kernel is the group of
 bisections $\Bis(\frac{P\times_{M} P}{H})$ of the associated Lie group bundle.
 The latter is a submanifold of $\Bis(\frac{P\times P}{H})$ since the adapted
 local addition on $\frac{P\times P}{H}$ used in the construction of the
 manifold structure on $\Bis(\frac{P\times P}{H})$ restricts to an adapted
 local addition on $\frac{P\times_{M} P}{H}$, and thus the corresponding carts
 for the manifold structure on $\Bis(\frac{P\times P}{H})$ yield submanifold
 charts for $\Bis(\frac{P\times_{M} P}{H})$. In total, we have a sequence of
 Lie groups
 \begin{equation}\label{eqn3}
  \Bis\left(\frac{P\times_{M} P}{H}\right)\hookrightarrow \Bis\left(\frac{P\times P}{H}\right)
  \twoheadrightarrow \im(\beta_{*}).
 \end{equation}
 In fact, we get this sequence if we start with an arbitrary locally trivial
 Lie groupoid \cite[p.\
 130]{Mackenzie05General-theory-of-Lie-groupoids-and-Lie-algebroids} (or in
 finite dimensions, equivalently with a Lie groupoid whose base is connected
 and whose Lie algebroid is transitive \cite[Corollary
 3.5.18]{Mackenzie05General-theory-of-Lie-groupoids-and-Lie-algebroids}).
 
 We will now explain how to turn this sequence into an extension of Lie groups,
 i.e., into a locally trivial bundle. To this end it suffices to construct a
 smooth section of $\beta_{*}$ on some identity neighbourhood of $\Diff(M)$.
 This then implies in particular that $\im(\beta_{*})$ is an open subgroup of
 $\Diff(M)$.
 
 Recall the notation from Example \ref{exmp:gauge_groupoid}. We choose a finite
 subcover $U_{1},...,U_{n}$ of the trivialising cover $(U_{i})_{i\in I}$. From
 Theorem \ref{thm: MFDMAP} \ref{thm:manifold_structure_on_smooth_mapping_a}
 recall the chart
 \begin{equation*}
  \varphi_{\id}\from O_{\id}\to \Gamma(M\xleftarrow{}TM)\cap C^{\infty}(M,U)
 \end{equation*}
 of $\Diff(M)$, where $U$ denotes the domain of a local addition
 $\A\from U\opn TM\to M$. Observe that
 $\varphi_{\id}^{-1}(h)(x)=\varphi_{\id}^{-1}(h')(x)$ if $h(x)=h'(x)$ follows
 from the construction of $\varphi_{\id}$. We now choose a partition of unity
 $\lambda_{i}\from M\to \R$ with $\supp(\lambda_i)\se U_{i}$. For
 $f\in O_{\id}$ we then have that
 \begin{equation*}
  s_{i}(f):=\varphi_{\id}^{-1}((\lambda_{1}+...+\lambda_{i-1})\cdot\varphi_{\id}(f))^{-1} \circ \varphi_{\id}^{-1}((\lambda_{1}+...+\lambda_{i})\cdot \varphi_{\id}(f))
 \end{equation*}
 defines a smooth map $s_{i}\from O_{\id}\to \Diff(M)$. Moreover, we have
 $s_{i}(f)(x)=x$ if $x\notin \supp(\lambda_{i})$, since
 \begin{equation*}
  ((\lambda_{1}+...+\lambda_{i-1}) \cdot\varphi_{\id}(f))(x)=((\lambda_{1}+...+\lambda_{i})\cdot\varphi_{\id}(f))(x)\quad\text{ for }\quad x\notin \supp(\lambda_{i}).
 \end{equation*}
 In addition, $s_{1}(f) \circ...\circ s_{n}(f)=f$ follows directly from the
 definition (see also \cite[Proposition
 1]{HallerTeichmann03Smooth-perfectness-through-decomposition-of-diffeomorphisms-into-fiber-preserving-ones}).
 
 With the aid of the chart
 \begin{equation*}
  K_{ij}\from \frac{\pi^{-1}(U_{i})\times \pi^{-1}(U_{j})}{H}\to U_{i}\times U_{j}\times H
 \end{equation*}
 we then define the bisection
 \begin{equation*}
  \wt{s}_{ij}(f)\from M\times M\to \frac{P\times P}{H},\quad (x,y)\mapsto  \begin{cases}
  K_{ij}^{-1}( (x,s_{j}(f)(y),e)) & \text{ if } (x,y)\in U_{i}\times U_{j}\\
  1_{(x,y)}=K_{ij}^{-1}(x,y,e) & \text{ else}.
  \end{cases}
 \end{equation*}
 This defines a smooth map since $s_{j}(f)(x)=x$ if
 $x\notin \supp(\lambda_{j})$ and $M\times (M\setminus \supp(\lambda_{j}))$ is
 open in $M\times M$. Moreover,
 \begin{equation*}
  \wt{s}_{ij}\from O_{\id}\to \Bis\left(\frac{P\times P}{H}\right),\quad f\mapsto \wt{s}_{ij}(f)
 \end{equation*}
 is smooth by Theorem \ref{thm: MFDMAP}
 \ref{thm:manifold_structure_on_smooth_mapping_h}. From
 $\beta \circ K_{ij}^{-1}=\pr_{2}$ we infer
 $\beta_{*}(\wt{s}_{ij}(f))=s_{j}(f)$ and thus
 \begin{equation*}
  O\to \Bis\left(\frac{P\times P}{H}\right),\quad f\mapsto \wt{s}_{11}(f)\star...\star\wt{s}_{nn}(f)
 \end{equation*}
 is a smooth section of $\beta_{*}$, where $\star$ is the product \eqref{eq:
 BISGP1}. This turns \eqref{eqn3} into an extension of Lie groups.
 
 We now identify $\Bis\left(\frac{P\times P}{H}\right)$ with the group of
 bundle automorphism of $P$ via the group isomorphism
 \begin{equation*}
  \Aut(P)\to \Bis\left(\frac{P\times P}{H}\right),\quad
  f \mapsto \left( m\mapsto \langle \sigma_{i}(m),f(\sigma_{i}(m))\rangle \text{ if }x\in U_{i}\right).
 \end{equation*}
 Under this isomorphism the subgroup $\Bis(\frac{P\times_{M}P}{H})$ maps to the
 group of gauge transformations $\Gau(P)$. If we denote by
 $\Diff(M)_{[P]}=\im(\beta_{*})$ the open subgroup of diffeomorphisms of $M$
 that lift to bundle automorphisms, then we obtain the well-known extension of
 Lie groups
 \begin{equation*}
  \Gau(P)\to \Aut(P)\to\Diff(M)_{[P]}
 \end{equation*}
 from
 \cite{Wockel07Lie-group-structures-on-symmetry-groups-of-principal-bundles}.
 Moreover, the natural action
 \begin{equation*}
  \Aut(P)\times P\to P,\quad (f,p)\mapsto f(p)
 \end{equation*}
 is smooth with respect to this identification, since it can be identified
 (non-canonically) with the action of $\Bis(\frac{P\times P}{H})$ on the
 $\alpha$-fibre $\alpha^{-1}(m)=\frac{P_{m}\times P}{H}\cong P$ for each
 $m\in M$.
\end{example}

\section{The Lie algebra of the group of bisections}
\label{sec:the_lie_algebra_of_the_group_of_bisections}

\begin{tabsection}
 In this section the Lie algebra of the Lie group $\Bis (\cG)$ is computed for
 a Lie groupoid $\cG = (G \toto M)$. Throughout this section $\cG$ denotes a
 locally convex and locally metrisable Lie groupoid with compact space of
 objects $M$ that admits an adapted local addition.

 It will turn out that the Lie algebra of the group of bisections is naturally
 (anti-) isomorphic to the Lie algebra of sections of the Lie algebroid
 $L(\cG)$ associated to $\cG$ (see Section
 \ref{sec:locally_convex_lie_groupoids_and_lie_groups} for the corresponding
 notions). Before we compute the bracket, let us identify the tangent space
 $T_1 \Bis (\cG)$.
\end{tabsection}

\begin{remark}
 By construction $\Bis (\cG)$ is an open submanifold of
 $\Gamma (M \xleftarrow{\alpha} G)$. We first analyse the space
 $T_1 C^{\infty}(M,G)$. This is by Theorem
 \ref{thm:tangent_map_of_pull_back_and_push_forward} isomorphic to the space
 $\Gamma(1^{*}TG)$, the isomorphism given by restricting the vector bundle
 isomorphism
 \begin{equation*}
  \Phi_{M,G}\from T C^{\infty}(M,G)\to C^{\infty}(M,TG),\quad    \eqclass{t\mapsto \eta(t)}\mapsto
  \left(m\mapsto \eqclass{t\mapsto \eta^\wedge(t,m)}\right)
 \end{equation*}
 to $T_{1}C^{\infty}(M,G)$. Here we have identified tangent vectors in
 $C^{\infty}(M,G)$ with equivalence classes $\eqclass{\eta}$ of smooth curves
 $\eta\from \mathopen{]}-\varepsilon,\varepsilon\mathclose{[}\to C^{\infty}(M,G)$
 for some $\varepsilon>0$ \cite[Definition I.3.3]{neeb2006}. This isomorphism
 maps $T_{1}C^{\infty}(M,G)$ onto
 \begin{equation*}
  \left\{f\in C^{\infty}(M,TG)\mid f(m)\in T_{1_{m}}G \text{ for all }m\in M \right\},
 \end{equation*}
 and the latter space is naturally isomorphic to $\Gamma(1^{*}TG)$. If we
 restrict in $C^{\infty}(M,G)$ to the submanifold
 $\Gamma(M\xleftarrow{\alpha}G)$, then this isomorphism maps
 $T_{1}(\Gamma(M\xleftarrow{\alpha} G))$ onto
 \begin{equation*}
  \left\{f\in C^{\infty}(M,TG)\mid f(m)\in T^{\alpha}_{1_{m}}G \text{ for all }m\in M \right\},
 \end{equation*}
 which in turn is naturally isomorphic to $\Gamma (\Lf(\cG))$. In the sequel we
 will denote by
 \begin{equation}\label{eqn:natural_isomorphism}
  \varphi_{\cG}\from T_{1}\Bis(\cG)=\Lf(\Bis(\cG))\to \Gamma(1^{*}T^{\alpha}G)=\Gamma(\Lf(\cG)).
 \end{equation}
 the resulting isomorphism.
\end{remark}

\begin{tabsection}
 Following some preparations, we will prove in Theorem \ref{theorem: LAlg} that
 the Lie algebra bracket $\LB{}$ on $T_1 \Bis (\cG)$ is, with respect to the
 isomorphism $\varphi_{\cG}$, the negative of the bracket of the Lie algebroid
 associated to $\cG$. To compute the Lie bracket on $T_1 \Bis (\cG)$ we adapt
 an idea of Milnor. In \cite[p.
 1041]{Milnor84Remarks-on-infinite-dimensional-Lie-groups} a natural action of
 the diffeomorphism group was used to compute the Lie bracket of its Lie
 algebra. In the present context we exploit the natural action of the group of
 bisections via left-translations on the manifold of arrows $G$ from Remark \ref{rem:structure_on_bisections} \ref{rem:structure_on_bisections_c}.
\end{tabsection}

\begin{proposition}\label{proposition:
 related} Let $X$ be an element of $T_1 \Bis (\cG)$ and denote by $0$ the
 zero-section in $\Gamma (TG)$. Then the vector fields
 $\overrightarrow{\varphi_{\cG}(X)}$ (cf.\ \eqref{eq: RI:VF} and
 \eqref{eqn:natural_isomorphism}) and $X^\rho \times 0$ are $\gamma$-related,
 i.e.\ the diagram
 \begin{equation}\label{diag: action}
  \begin{aligned} \begin{xy}
  \xymatrix{
  T\Bis (\cG) \times TG \ar[rr]^-{T\gamma}  & &  TG  \\
  \Bis (\cG) \times G \ar[rr]^-{\gamma} \ar[u]^{X^\rho \times 0}  &       &   G \ar[u]_{\overrightarrow{\varphi_{\cG} (X)}}   
  }
  \end{xy} \end{aligned}
 \end{equation}
 commutes.
\end{proposition}

\begin{proof}
 To simplify computations we identify $X$ with the equivalence class
 $\eqclass{\eta}$ of a smooth curve
 $\eta \colon \mathopen{]}-\varepsilon , \varepsilon \mathclose{[} \rightarrow \Bis (\cG)$
 satisfying $\eta (0)= 1$ and $\eta' (0) = X$. From Theorem \ref{thm: MFDMAP}
 \ref{thm:manifold_structure_on_smooth_mapping_d} we infer that
 $\eta^{\wedge} \colon \mathopen{]}-\varepsilon , \varepsilon \mathclose{[} \times M \rightarrow G$,
 $(t,m)\mapsto \eta(t)(m)$ is smooth. Thus for each $\psi \in \Bis (\cG)$ we
 obtain the smooth map $(\rho_\psi \circ \eta)^{\wedge}$. \ Evaluating in
 $(t,x) \in \mathopen{]}-\varepsilon , \varepsilon \mathclose{[} \times M$ we
 obtain the formula
 \begin{equation}\label{eq: eval}
  (\rho_\psi \circ \eta)^{\wedge} (t,x) = (\eta^{\wedge} (t, \cdot) \star \psi)(x) = m (\eta^{\wedge} (t, \beta (\psi (x))),\psi(x)).
 \end{equation}
 Moreover, by definition of right invariant vector fields
 $X^{\rho} (\psi ) = \eqclass{t \mapsto \rho_\psi \circ\eta^{\wedge} (t ,\cdot)}$
 holds for each $\psi \in \Bis (\cG)$. We use the above facts to compute for
 $(\psi,g) \in \Bis (\cG) \times G$
 \begin{align*}
  \overrightarrow{\varphi_{\cG} (X)} \circ \gamma (\psi,g) &\stackrel{\hphantom{\eqref{eq: eval}}}{=}  
  \overrightarrow{\varphi_{\cG} (X)}(m(\psi (\beta (g)),g))   \stackrel{\eqref{eq: RI:VF}}{=}  
  TR_{m(\psi (\beta (g)),g)} \varphi_{\cG} (X) (\beta (m(\psi (\beta (g)),g))) \\  &\stackrel{\hphantom{\eqref{eq: eval}}}{=}  
  TR_{m(\psi (\beta (g)),g)} \varphi_{\cG}( X) (\beta (\psi (\beta (g)))) = 
  \eqclass{t \mapsto m (m (\eta^{\wedge} (t, \beta (\psi(\beta (g)))), \psi (\beta (g))),g)}\\
  &\stackrel{\eqref{eq: eval}}{=} 
  \eqclass{t \mapsto m(\rho_\psi \circ \eta^{\wedge} (t,\beta(g)),g)} = 
  \eqclass{t \mapsto \gamma (\rho_\psi \circ \eta^{\wedge} (t,\cdot),g)} \\
  &\stackrel{\hphantom{\eqref{eq: eval}}}{=} 
  T\gamma (X^\rho (\psi) , 0_g) = 
  T\gamma \circ (X^\rho \times 0) (\psi , g).
 \end{align*}
 Hence \eqref{diag: action} commutes and the assertion follows.
\end{proof}

\begin{setup}
 Before phrasing the main result of this section we introduce the following
 notation. If we fix a manifold $M$ and consider the category
 $\cat{LieAlgebroids}_{M}$ of locally convex Lie algebroids over $M$, then
 taking sections gives rise to a functor
 \begin{equation*}
  \Gamma\from \cat{LieAlgebroids}_{M}\to \cat{LieAlgebras}.
 \end{equation*}
 Likewise, there is the functor
 \begin{equation*}
  -\Gamma\from \cat{LieAlgebroids}_{M}\to \cat{LieAlgebras}
 \end{equation*}
 which assigns to a Lie algebroid its Lie algebra of sections, but with the
 negative Lie bracket on it.
\end{setup}

\begin{theorem}\label{theorem:
 LAlg} Let $M$ be a compact manifold and $\cG=(G\toto M)$ be a locally convex
 Lie groupoid admitting an adapted local addition. Then the morphism of
 topological vector spaces
 \begin{equation*}
  \varphi_{\cG}\from \Lf(\Bis(\cG))\to \Gamma(\Lf(\cG))
 \end{equation*}
 from \eqref{eqn:natural_isomorphism} is actually an anti-isomorphism of Lie
 algebras. Moreover, $\varphi_{\cG}$ constitutes a natural isomorphism fitting
 into the diagram
 \begin{equation*}
  \vcenter{ \xymatrix{\cat{LieGroupoids}_{M}^{\A} \ar[r]^{\Lf}\ar[d]^{\Bis} & \cat{LieAlgebroids}_{M}\ar[d]^{-\Gamma} %
  \\ \cat{LieGroups} \ar[r]^{\Lf}&\cat{LieAlgebras},
  \ar@{=>}^{\varphi} "2,1"+(10,4); "1,2"-(10,4)
  }}
 \end{equation*}
 of functors.
\end{theorem}

\begin{proof}
 Recall that the bracket on $\Gamma (\Lf(\cG))$ is induced from the isomorphism
 \eqref{eq: RI:VF} with the right invariant vector fields on $G$. In
 Proposition \ref{proposition: related} we have seen that for
 $X,Y\in T_{1}\Bis(\cG)$ the right-invariant vector fields
 $\overrightarrow{\varphi_{\cG} (X)}$ and $\overrightarrow{\varphi_{\cG} (Y)}$
 are $\gamma$-related to $X^\rho \times 0$ and $Y^\rho \times 0$, respectively.
 Hence the Lie bracket
 $\LB[X^\rho \times 0 , Y^\rho \times 0] = \LB[X^\rho , Y^\rho] \times 0$ is
 $\gamma$-related to the Lie bracket
 $\LB[\overrightarrow{\varphi_{\cG}( X)},\overrightarrow{\varphi_{\cG} (Y)}]$,
 i.e., we have
 \begin{equation*}
  T\gamma \circ ([X^\rho,Y^\rho] \times 0) = [\overrightarrow{\varphi _{\cG}(X)} , \overrightarrow{\varphi_{\cG} (Y) }] \circ \gamma.
 \end{equation*}
 From this we deduce
 \begin{align*}
  -\varphi_{\cG}([X,Y])(x)&\stackrel{\hphantom{\ref{setup: RI:VF}}}{=}-\varphi_{\cG}(\overrightarrow{[X,Y]})(1_{x}) =-\varphi_{\cG}(\overrightarrow{[X,Y]})(\gamma(1,1_{x}))=T \gamma((-[X,Y]^{\rho}\times 0)(1,1_{x}))\\
  &\stackrel{\ref{setup: RI:VF}}{=} T \gamma(([X^{\rho},Y^{\rho}]\times 0)(1,1_{x})) =[\overrightarrow{\varphi_{\cG} (X)} , \overrightarrow{\varphi_{\cG}( Y)}] (\gamma(1,1_{x}))=[\overrightarrow{\varphi_{\cG(} X)} , \overrightarrow{\varphi_{\cG} (Y}] (1_{x})\\
  &\stackrel{\hphantom{\ref{setup: RI:VF}}}{=}[\overrightarrow{\varphi_{\cG} (X)} , \overrightarrow{\varphi_{\cG}( Y)}] ({x}).
 \end{align*}
 Hence \eqref{eq: LB} implies that $\varphi_{\cG}$ is an anti-isomorphism.
 
 To check that $\varphi_{\cG}$ is natural, suppose $\cH=(H\toto M)$ is another
 Lie groupoid over $M$ admitting an adapted local addition and
 $f\from \cG\to \cH$ is a morphism. Then $-\Gamma(\Lf(f))$ is the induced map
 on sections $\Gamma(\Lf(\cG))\to \Gamma(\Lf(\cH))$,
 $\xi\mapsto T f \circ \xi$. On the other hand, $\Bis(f)$ is the map
 $\Bis(\cG)\to \Bis(\cH)$, $\sigma\mapsto f \circ \sigma$. The tangent map of
 $f_{*}\from C^{\infty}(M,G)\to C^{\infty}(M,H)$,
 $\sigma \mapsto f \circ \sigma$ at $1$ is given by Theorem
 \ref{thm:tangent_map_of_pull_back_and_push_forward} by
 \begin{equation}\label{eqn2}
  T_{1}(f_{*})\from \Gamma(1^{*}TG)\to \Gamma(1^{*}TH),\quad
  \xi\mapsto Tf \circ \xi
 \end{equation}
 with respect to the identifications
 $T_{1}C^{\infty}(M,G)\cong \Gamma(1^{*}TG)$ and
 $T_{1}C^{\infty}(M,H)\cong \Gamma(1^{*}TH)$. Restricting the latter
 isomorphism to vertical sections gives exactly the isomorphism $\varphi_{\cG}$
 and \eqref{eqn2} gives the above formula for $-\Gamma(\Lf(f))$.
\end{proof}

\begin{example}
 For the gauge groupoid $\frac{P\times P}{H}\toto M$ of the principal
 $H$-bundle $\pi\from P\to M$ with locally exponential structure group $H$ we
 have the natural isomorphisms $\Lf(\frac{P\times P}{H}\toto M)\cong TP/H\to M$
 from Example \ref{exmp:gauge_groupoid}. Thus the extension
 \begin{equation}
  \Bis\left(\frac{P\times_{M} P}{H}\right)\hookrightarrow \Bis\left(\frac{P\times P}{H}\right)
  \twoheadrightarrow \im(\beta_{*})
 \end{equation}
 of locally convex Lie groups from \eqref{eqn3} gives rise via the latter
 isomorphism and the isomorphism from \eqref{eqn:natural_isomorphism} to the
 extension
 \begin{equation*}
  \Gamma(M\xleftarrow{} (T^{v}P)/H) \to \Gamma(M\xleftarrow{} {TP/H}) \to \Gamma(M\xleftarrow{} {TM})
 \end{equation*}
 of topological Lie algebras, where $T^{v}P:= \ker(T \pi)$ denotes the vertical
 subalgebroid of $TP/H$. This is of course the extension of Lie algebras which
 is naturally associated to the Atiyah sequence $T^{v}P/H \to TP/H \to TM$.
\end{example}

\begin{setup}
 For the following corollary, recall that a Lie algebroid $\cA$ is called
 \emph{integrable} if there is a Lie groupoid $\cG$ such that $\Lf(\cG)$ is
 isomorphic (over $M$) to $\cA$. In the same way, a Lie algebra $\fh$ is called
 \emph{integrable} if there is a Lie group $H$ such that $\Lf(H)$ is isomorphic
 to $\fh$.
\end{setup}

\begin{corollary}
 Suppose $M$ is a compact manifold and $\cA=(A\to M,a,\LB{})$ is a
 finite-dimensional Lie algebroid over $M$. If $\cA$ is integrable, then so is
 its algebra of sections $(\Gamma(M\xleftarrow{}A),\LB)$.
\end{corollary}

\begin{question}
 The previous result is not a surprise. What is more interesting is the
 question about the converse statement: suppose that a finite-dimensional Lie
 algebroid is \emph{not} integrable, is then its algebra of sections also not
 integrable?
\end{question}

\begin{remark}
 Note that the Lie algebras that arise here as the Lie algebras of Lie groups
 of bisections carry more information that just the structure of a Lie algebra.
 In fact, the geometric structure that they have is subsumed in the notion of a
 Lie-Rinehart algebra
 \cite{Huebschmann90Poisson-cohomology-and-quantization,Kosmann-SchwarzbachMagri90Poisson-Nijenhuis-structures}.
 Thus a way to solve the above question could be to create a theory of objects
 that are integrating Lie-Rinehart algebras on an algebraic level (something
 that one might call Lie-Rinehart groups). To our best knowledge such a theory
 does not exist at the moment.
\end{remark}

\section{Regularity properties of the group of bisections}
\label{sec:regularity_properties_of_the_group_of_bisections}

\begin{tabsection}
 This section contains an investigation of regularity properties for the Lie
 group of bisections. Throughout this section $\cG$ denotes a locally convex
 and locally metrisable Lie groupoid with compact space of objects $M$ that
 admits an adapted local addition. Moreover, we identify throughout this
 section the Lie algebra $\Lf(\Bis(\cG))$ with $-\Gamma(\Lf(\cG))$ via the
 isomorphism $\varphi_{\cG}$ from Theorem \ref{theorem: LAlg}.
 
 We will give two completely different proofs of the $C^{k}$-regularity of
 $\Bis(\cG)$ in the case of a locally trivial Lie groupoid and in the case of a Banach-Lie
 groupoid. While the argument in the locally continuous case is geometric in
 nature (and rather elementary), the argument in the case of Banach-Lie
 groupoids is analytical.
\end{tabsection}

\begin{theorem}\label{thm:ltriv_reg}
 Let $\cG = (G \toto M)$ be a Lie groupoid and $k \in \N_0 \cup \{\infty\}$.
 Assume that $G$ is a locally trivial Lie groupoid with locally exponential and
 $C^{k}$-regular vertex group and compact $M$. Then the Lie group $\Bis (\cG)$
 is $C^k$-regular. In particular, the Lie group $\Bis (\cG)$ is regular in the
 sense of Milnor.
\end{theorem}

\begin{proof}
 First note that $\cG$ is isomorphic over $M$ to the gauge groupoid of a
 principal $K$-bundle $P\to M$, where $K$ is the vertex group of $\cG$. So we
 may assume without loss of generality that $\cG=(\frac{P\times P}{H}\toto M)$.
 We consider the extension
 \begin{equation*}
  \Bis\left(\frac{P\times_{M} P}{H}\right)\hookrightarrow \Bis\left(\frac{P\times P}{H}\right)
  \twoheadrightarrow \im(\beta_{*})
 \end{equation*}
 of Lie groups from Example \ref{exmp:gauge_and_automorphism_group}. As
 explained in Example \ref{exmp:gauge_and_automorphism_group} the Lie group
 $\Bis\left(\frac{P\times_{M} P}{H}\right)$ is isomorphic to the gauge group
 $\Gau(P)$. This isomorphism is even an isomorphism of locally metrisable Lie
 groups since it maps smooth curves to smooth curves. From
 \cite{Glockner13Regularity-properties-of-infinite-dimensional-Lie-groups} it
 now follows that $\Bis\left(\frac{P\times_{M} P}{H}\right)$ is
 $C^{k}$-regular, as well as $\im(\beta_{*})$ (the latter is just an open
 subgroup of $\Diff(M)$). Since $C^{k}$-regularity is an extension property
 \cite[Appendix
 B]{NeebSalmasian12Differentiable-vectors-and-unitary-representations-of-Frechet-Lie-supergroups}
 it follows that also $\Bis\left(\frac{P\times P}{H}\right)$ is
 $C^{k}$-regular, what we were after to show.
\end{proof}

\begin{setup}\label{setup:
 diffeq} Let $\cG$ be a Lie groupoid. Define the map
 \begin{displaymath}
  f \colon [0,1] \times G \times C^0 ([0,1] , \Gamma (\Lf(\cG)))_{\text{c.o.}} \rightarrow
  T^\alpha G,\quad (t,g, \eta) \mapsto TR_g \eta^\wedge (t,\beta (g)) \coloneq
  TR_g\eta (t) (\beta (g)).
 \end{displaymath}
 This map makes sense, since \eqref{eq: RI:VF} shows for fixed
 $(t,\eta) \in [0,1] \times C^0 ([0,1] , \Gamma (\Lf(\cG)))$ that
 $f(t,\cdot, \eta)$ is the right-invariant vector field associated to
 $\eta (t)$ and thus takes its values in $T^\alpha G$. We now consider the
 parameter dependent initial value problem:
 \begin{equation}\label{eq:diffeq}
  \begin{cases}
  x'(t)  &= f (t,x(t),\eta) =TR_{x(t)} \eta^\wedge (t,\beta (x(t))),\\
  x(t_0) &= g_0, \quad \quad (t_0,g_0) \in [0,1] \times G
  \end{cases}
 \end{equation}
 To prove the regularity of $\Bis (\cG)$ we will study the flow of the
 differential equation.
\end{setup}

Recall the following definition of $C^{r,s}$-mappings from \cite{alas2012}.

\begin{setup}
 Let $E_1$, $E_2$ and $F$ be locally convex spaces, $U$ and $V$ open subsets of
 $E_1$ and $E_2$, respectively, and $r,s \in \N_0 \cup \{\infty\}$. A mapping
 $f\colon U \times V \rightarrow F$ is called a $C^{r,s}$-map if for all
 $i,j \in \N_0$ such that $i \leq r, j \leq s$, the iterated directional
 derivative
 \begin{displaymath}
  d^{(i,j)}f(x,y,w_1,\dots,w_i,v_1,\dots,v_j) \coloneq (D_{(w_i,0)} \cdots
  D_{(w_1,0)}D_{(0,v_j)} \cdots D_{(0,v_1)}f ) (x,y)
 \end{displaymath}
 exists for all
 $ x \in U, y \in V, w_1, \ldots , w_i \in E_1,  v_1, \ldots ,v_j \in E_2$ and
 yields continuous maps
 \begin{align*} 
  d^{(i,j)}f\colon   & U \times V \times E^i_1 \times E^j_2 \rightarrow F,\\ 
  &(x,y,w_1,\dots,w_i,v_1,\dots,v_j)\mapsto (D_{(w_i,0)} \cdots D_{(w_1,0)}D_{(0,v_j)} \cdots D_{(0,v_1)}f ) (x,y).
 \end{align*}
 One can extend the definition of $C^{r,s}$-maps to mappings on locally convex
 domains with dense interior (cf.\ Definition \ref{defn: nonopen}). 
 In addition, there are chain rules for $C^{r,s}$-mappings allowing us to naturally extend the notion of $C^{r,s}$-maps
 to maps defined on products of locally convex manifolds with values in a locally convex manifold.  
\end{setup}

\begin{tabsection}
 For further results and details on the calculus of $C^{r,s}$-maps we refer to
 \cite{alas2012}.
\end{tabsection}

\begin{proposition}\label{proposition:
 flow}
 \begin{enumerate}
  \item \label{proposition: flow_a} The map $f$ from \ref{setup: diffeq} is of
        class $C^{0,\infty}$ with respect to the splitting\\
        $[0,1] \times (G \times C^0 ([0,1] , \Gamma (\Lf(\cG)))_{\text{c.o.}})$ and satisfies
        \begin{displaymath}
         f(t,g,\eta) \in T_g^\alpha G \text{ for }(t,g,\eta)\in [0,1] \times G
         \times C^0 ([0,1] , \Gamma (\Lf(\cG))).
        \end{displaymath}
 \end{enumerate}
 \noindent Assume in addition that $G$ is a Banach-manifold. Then the following
 holds:
 \begin{enumerate}\setcounter{enumi}{1}
  \item \label{proposition: flow_b} There is a zero-neighbourhood
        $\Omega \opn \Gamma (\Lf(\cG))$ such that for every
        $(t_0, g_0 ,\eta) \in [0,1] \times G \times C^0 ([0,1] , \Omega)$ the
        initial value problem \eqref{eq:diffeq} admits a unique maximal
        solution $\varphi_{t_0,g_0,\eta} \colon [0,1] \rightarrow G$. Here we
        defined
        $C^0([0,1] ,\Omega) \coloneq \left\{X \in C^0 ([0,1] , \Gamma (\Lf(\cG))) \middle| X([0,1])\subseteq \Omega \right\}\opn C^0 ([0,1] , \Gamma (\Lf(\cG)))_{\text{c.o.}}$.\\
        Hence, we obtain the \emph{flow} of \eqref{eq:diffeq} as
        \begin{displaymath}
         \Fl^f \colon [0,1] \times [0,1] \times G \times C^0([0,1] ,\Omega)
         \rightarrow G, (t_0,t,g_0,\eta) \mapsto \varphi_{t_0,g_0,\eta} (t).
        \end{displaymath}
  \item \label{proposition: flow_c} The map
        $\Fl^f_0 \coloneq \Fl^f (0,\cdot ) \colon [0,1] \times (G \times C^0([0,1] ,\Omega)) \rightarrow G, (t,g,\eta) \mapsto \Fl^f (0,t,g,\eta)$
        is of class $C^{1,\infty}$.
  \item \label{proposition: flow_d}Fix
        $(s,t,\eta) \in [0,1] \times [0,1] \times C^0([0,1] ,\Omega)$ then the
        map
        $\beta \circ \Fl^f (s,t, \cdot , \eta) \circ 1 \colon M \rightarrow M$
        is a diffeomorphism.
  \item \label{proposition: flow_e} Fix $\eta \in C^0 ([0,1] , \Omega)$, then
        $H_\eta \colon [0,1] \times M \rightarrow G , (t,x) \mapsto \Fl^f_0 (t, 1_x, \eta)$
        is a $C^{1,\infty}$-mapping which induces a $C^1$-map
        \begin{displaymath}
         c_\eta \colon [0,1] \rightarrow \Bis (\cG) , t \mapsto \Fl^f_0 (t,
         \cdot,\eta).
        \end{displaymath}
 \end{enumerate}
\end{proposition}

We postpone the rather technical proof of Proposition \ref{proposition: flow}
to Section \ref{Appendix: ODE}.

\begin{theorem}\label{theorem:
 Banach:reg} Let $\cG = (G \toto M)$ be a Lie groupoid and assume that $G$ is a
 Banach manifold and that $M$ is compact. Then the Lie group $\Bis (\cG)$ is
 $C^k$-regular for each $k \in \N_0 \cup \{\infty\}$. In particular, the Lie group
 $\Bis (\cG)$ is regular in the sense of Milnor.
\end{theorem}

\begin{proof}
 Let $\Omega \opn \Gamma (\Lf(\cG)) = T_1 \Bis (\cG) = \Lf(\Bis (\cG))$ be the
 zero-neighbourhood constructed in Proposition \ref{proposition: flow}
 \ref{proposition: flow_b}. Combine Proposition \ref{proposition: flow}
 \ref{proposition: flow_c} and \ref{proposition: flow_d} with Theorem \ref{thm:
 MFDMAP} \ref{thm:manifold_structure_on_smooth_mapping_d} to obtain a smooth
 map
 \begin{displaymath}
  \evol \colon C^0 ([0,1] , \Omega) \rightarrow \Bis (\cG) ,\quad \eta
  \mapsto \Fl^f_0 (1, \cdot, \eta ) \circ 1 = c_\eta (1).
 \end{displaymath}
 We claim that $c_\eta \colon [0,1] \rightarrow \Bis (\cG)$ is the product
 integral of $\eta \colon [0,1] \rightarrow \Omega \subseteq \Lf(\Bis (\cG))$,
 i.e.\ it solves the initial value problem (cf.\ \eqref{eq: regular})
 \begin{displaymath}
  \begin{cases}
   \gamma' (t) &= T_1 \rho_{\gamma (t)} (\eta (t)) = \eta (t).\gamma(t)\\
   \gamma (0) &= 1
  \end{cases}
  .
 \end{displaymath}
 If this is true, then the proof can be completed as follows: For each
 $\eta \in C^0([0,1], \Omega) \opn C^0([0,1], \Gamma (\Lf(\cG)))$ there is a
 product integral and the evolution $\evol$ is smooth. Then $\Bis (\cG)$ is
 $C^0$-regular by \cite[Proposition 1.3.10]{Dahmen2012}. Since $C^0$-regularity
 implies $C^k$-regularity for all $k \geq 0$ the assertion follows.
 
 \textbf{Proof of the claim:} Fix $\eta \in C^0([0,1] , \Omega)$ and observe
 that $c_\eta \colon [0,1] \rightarrow \Bis (\cG)$ is a $C^1$-curve by
 Proposition \ref{proposition: flow} \ref{proposition: flow_e}. Furthermore,
 $c_\eta (0) = H_\eta (0,\cdot) = \Fl^f_0 (0, \cdot,\eta) \circ 1 = 1 \in \Bis (\cG)$.
 Let us now compute the derivative $\frac{\partial}{\partial t} c_\eta (t)$ for
 fixed $t \in [0,1]$. To this end choose a smooth curve
 $k \colon ]-\varepsilon,\varepsilon[ \rightarrow \Bis (\cG)$ (for some
 $\varepsilon >0$) with $k(0)=1$ and $k'(0)=\eta(t) \in T_1 \Bis (\cG)$. Recall
 that
 $\Phi_{M,G} \colon TC^\infty (M,G) \rightarrow C^\infty(M,TG), \eqclass{t\mapsto h(t)} \mapsto (m\mapsto \eqclass{t\mapsto h^\wedge (t,m)})$
 is an isomorphism of vector bundles by Theorem \ref{thm:tangent_map_of_pull_back_and_push_forward}.
 Therefore, we can
 compute the derivative as follows:
 \begin{align*}
  \frac{\partial}{\partial t} c_\eta (t) &\stackrel{\hphantom{\eqref{eq:diffeq}}}{=} \eqclass{s \mapsto c_\eta (t+s)} = \Phi^{-1}_{M,G} \left(m\mapsto \eqclass{s\mapsto \Fl^f_0 (t+s,1_m,\eta)}\right)\\
  &\stackrel{\eqref{eq:diffeq}}{=} \Phi_{M,G}^{-1} \left(m \mapsto TR_{\Fl^f_0 (t,1_m,\eta)} \eta^\wedge (t,\beta (\Fl^f_0 (t,1_m,\eta)))\right)\\
  &\stackrel{\hphantom{\eqref{eq:diffeq}}}{=}  \Phi_{M,G}^{-1} \left(m \mapsto TR_{c_\eta^\wedge (t,m)} \circ \eta^\wedge (t,\cdot) \circ (\beta \circ c_\eta)^\wedge (t,m)))\right)\\
  &\stackrel{\hphantom{\eqref{eq:diffeq}}}{=}  \Phi^{-1}_{M,G} \left(m\mapsto \eqclass{s \mapsto R_{c_\eta^\wedge (t,m)} \circ k^\wedge (s,\cdot) \circ \beta \circ c_\eta^\wedge (t,m)}\right) \\
  &\stackrel{\hspace{2pt}\eqref{eq: BISGP1}\hspace{2pt}}{=}  \Phi^{-1}_{M,G} \left(m\mapsto \eqclass{s \mapsto ((\rho_{c_\eta (t)} \circ k)^\wedge (s,m)}\right) = \eqclass{s \mapsto \rho_{c_\eta (t)} \circ k (s)} = T_1 \rho_{c_\eta (t)} (\eqclass{s\mapsto k (s)})\\ 
  &\stackrel{\hphantom{\eqref{eq:diffeq}}}{=} T_1 \rho_{c_\eta (t)} \eta (t)
 \end{align*}
 As $t \in [0,1]$ was arbitrary, $c_\eta$ is the product integral for
 $\eta \colon [0,1] \rightarrow \Lf(\Bis (\cG))$.
\end{proof}

\section{Proof of Proposition \ref{proposition: flow}}\label{Appendix: ODE}

In this section we exhibit the technical proof of Proposition
\ref{proposition: flow}. Let us first recall its content:

\begin{setup}
 \begin{enumerate}
  \item The map $f$ from \ref{setup: diffeq} is of class $C^{0,\infty}$ with
        respect to the splitting
        $[0,1] \times (G \times C^0 ([0,1] , \Gamma (\Lf(\cG)))_{\text{c.o.}})$ and satisfies
        \begin{displaymath}
         f(t,g,\eta) \in T_g^\alpha G \text{ for }(t,g,\eta)\in [0,1] \times G \times C^0 ([0,1] , \Gamma (\Lf(\cG))).
        \end{displaymath}
 \end{enumerate}\noindent
 Assume in addition that $G$ is a Banach-manifold. Then the following holds:
 \begin{enumerate}
  \item[b)] There is a zero-neighbourhood $\Omega \opn \Gamma (\Lf(\cG))$
        such that for every
        $(t_0, g_0 ,\eta) \in [0,1] \times G \times C^0 ([0,1] , \Omega)$ the
        initial value problem \eqref{eq:diffeq} admits a unique maximal
        solution $\varphi_{t_0,g_0,\eta} \colon [0,1] \rightarrow G$. Here we
        defined
        $C^0([0,1] ,\Omega) \coloneq \left\{X \in C^0 ([0,1] , \Gamma (\Lf(\cG))) \middle| X([0,1])\subseteq \Omega \right\}\opn C^0 ([0,1] , \Gamma (\Lf(\cG)))_{\text{c.o.}}$.\\
        Hence, we obtain the \emph{flow} of \eqref{eq:diffeq} as
        \begin{displaymath}
         \Fl^f \colon [0,1] \times [0,1] \times G \times  C^0([0,1] ,\Omega) \rightarrow G, (t_0,t,g_0,\eta) \mapsto \varphi_{t_0,g_0,\eta} (t).
        \end{displaymath}
  \item[c)] The map
        $\Fl^f_0 \coloneq \Fl^f (0,\cdot ) \colon [0,1] \times (G \times C^0([0,1] ,\Omega)) \rightarrow G, (t,g,\eta) \mapsto \Fl^f (0,t,g,\eta)$
        is of class $C^{1,\infty}$.
  \item[d)] Fix
        $(s,t,\eta) \in [0,1] \times [0,1] \times C^0([0,1] ,\Omega)$ then the
        map
        $\beta \circ \Fl^f (s,t, \cdot , \eta) \circ 1 \colon M \rightarrow M$
        is a diffeomorphism.
  \item[e)] Fix $\eta \in C^0 ([0,1] , \Omega)$, then
        $H_\eta \colon [0,1] \times M \rightarrow G , (t,x) \mapsto \Fl^f_0 (t, 1_x, \eta)$
        is a $C^{1,\infty}$-mapping which induces a $C^1$-map
        \begin{displaymath}
         c_\eta \colon [0,1] \rightarrow \Bis (\cG) , t \mapsto \Fl^f_0 (t, \cdot,\eta).
        \end{displaymath}
 \end{enumerate}
\end{setup}

\begin{proof}[of
 Proposition \ref{proposition: flow}]
 \begin{enumerate}
  \item To prove that $f$ is a $C^{0,\infty}$-map, define first an auxiliary
        map
        \begin{displaymath}
         f_0 \colon [0,1] \times M \times C^0 ([0,1], \Gamma
         (\Lf(\cG)))_{\text{c.o.}} \rightarrow T^\alpha G , (t,x,\eta) \mapsto
         \eta^\wedge (t,x) = \eta (t) (x).
        \end{displaymath}
        We will first prove that $f_0$ is of class $C^{0,\infty}$ with respect
        to the splitting
        $[0,1] \times (M \times C^r ([0,1] , \Gamma (\Lf(\cG)))_{\text{c.o.}})$.
        To this end consider the evaluation maps
        $\op{ev}_0 \colon C^0 ([0,1], \Gamma (\Lf(\cG)))_{\text{c.o.}} \times [0,1] \rightarrow \Gamma (\Lf(\cG)) , \op{ev}_0 (\eta, t) = \eta (t)$
        and
        $\widetilde{\op{ev}} \colon \Gamma (\Lf(\cG)) \times M \rightarrow T^\alpha G, \widetilde{\op{ev}} (X,y) = X(y)$.
        Clearly
        $f_0 (t,x,\eta) = \widetilde{\op{ev}} \circ (\op{ev}_0 (\eta ,t), g)$.
        Since $\op{ev}_0$ is a $C^{\infty,0}$-map by \cite[Proposition
        3.20]{alas2012}, a combination of the chain rules \cite[Lemma 3.17 and
        Lemma 3.19]{alas2012} for $C^{r,s}$-mappings shows that $f_0$ will be
        $C^{0,\infty}$ if $\widetilde{\op{ev}}$ is smooth. To see that
        $\widetilde{\op{ev}}$ is smooth we compute in bundle charts. Consider
        the vector bundle $\pi_\alpha \colon T^\alpha G \rightarrow G$ and
        denote its typical fibre by $E$, We choose a local trivialisation
        $\kappa \colon \pi_{\alpha}^{-1} (U_\kappa) \rightarrow U_\kappa \times E$
        such that $U_\kappa \cap 1 (M) \neq \emptyset$. By construction
        \ref{setup: algebroid}, the vector bundle $\Lf(G)\to M$ is the pullback
        bundle of $T^\alpha G$ over the embedding $1$. Hence $\kappa$ induces
        the trivialisation
        $1^*\kappa \colon (1*\pi_\alpha)^{-1} (1^{-1} (U_\kappa)) \rightarrow 1^{-1} (U_\kappa) \times E, Y \mapsto (1^*\pi_\alpha (Y), \kappa (\pi_\alpha^* 1 (Y)))$
        of $\Lf(G)\to M$. Shrinking $U_\kappa$ we may assume that
        $W \coloneq 1^{-1} (U_\kappa)$ is the domain of a manifold chart
        $(\psi, W)$ of $M$. Recall from \cite[Proposition 7.3 and Lemma
        5.5]{Wockel13Infinite-dimensional-and-higher-structures-in-differential-geometry}
        that the map
        $\theta_{\kappa,\psi} \colon \Gamma (\Lf(\cG)) \rightarrow C^\infty (\psi (W) , E) , X\mapsto \op{pr}_2 \circ 1^*\kappa \circ X \circ \psi^{-1}$
        is continuous linear, whence smooth. We obtain a commutative diagram
        with smooth columns
        \begin{equation*}
         \begin{aligned} \begin{xy}
         \xymatrix{
         \Gamma (\Lf(\cG)) \times W \ar[d]_{(\theta_{\kappa,\psi} \times \psi)} \ar[rr]^-{\widetilde{\op{ev}}|_{\Gamma (\Lf(\cG)) \times W}^{\pi_\alpha^{-1} (U_\kappa)}} &  &  \pi_\alpha^{-1} (U_\kappa) \\
         C^\infty(\psi(W),E) \times \psi (W) \ar[rr]^-{(1 \circ \op{pr}_2 , \op{ev})} & & U_\kappa \times E \ar[u]_{\kappa^{-1}}
         }
         \end{xy} \end{aligned}
        \end{equation*}
        where
        $\op{ev} \colon C^\infty (\psi (W) , E) \times \psi (W) \rightarrow E, (\lambda,x) \mapsto \lambda (x)$
        is the evaluation map. By \cite[Proposition 3.20]{alas2012} (which is
        applicable by \cite[Lemma
        5.3]{Wockel13Infinite-dimensional-and-higher-structures-in-differential-geometry})
        the map $\op{ev}$ is smooth. Furthermore, the local trivialisation
        $\kappa$ was chosen arbitrarily, the map $\widetilde{\op{ev}}$ is
        smooth and in conclusion $f_0$ is of class $C^{0,\infty}$.
        
        We claim
        that the following diagram makes sense and commutes:
        \begin{equation}\label{diag: smooth:ev}
         \begin{aligned} \begin{xy}
         \xymatrix{
         [0,1] \times G \times C^0 ([0,1], \Gamma (\Lf(\cG)))_{\text{c.o.}} \ar[rr]^-{f} \ar[d]_{(\id_{[0,1]} \times (\beta , \id_{G}) \times \id_{C^r([0,1],\Gamma (\Lf(\cG)))})} &  & TG \\
         [0,1] \times M \times C^0 ([0,1], \Gamma (\Lf(\cG)))_{\text{c.o.}} \times G \ar[rr]^-{f_0 \times 0} & & TG \times_{T\alpha,T\beta}TG \ar[u]_{Tm}
         }
         \end{xy} \end{aligned}
        \end{equation}
        Assume for a moment that the claim is true and \eqref{diag: smooth:ev}
        commutes. Then $f$ is a $C^{0,\infty}$-map by the chain rules for
        $C^{0,\infty}$ mappings and smooth maps \cite[Lemma 3.17 and Lemma
        3.18]{alas2012}. By construction (see \ref{setup: diffeq}) the map $f$
        factors through the split submanifold $T^\alpha G$, whence assertion
        a) follows.\bigskip
        
        \textbf{Proof of the claim:} Fix
        $(t,g,\eta) \in [0,1] \times G \times C^0([0,1], \Gamma (\Lf(\cG)))$.
        Notice first that the composition makes sense: By definition we have
        $f_0 (t,\beta (g),\eta) = T^\alpha_{1(\beta (g))} G$. Hence
        $T\alpha (f_0 (t,\beta (g),\eta)) = 0_{\beta (g)} = T\beta (0(g))$. In
        conclusion $f_0 \times 0$ factors through
        $TG \times_{T\alpha,T\beta}TG$. We are left to prove that \eqref{diag:
        smooth:ev} commutes. To see this we will use the explicit formula for
        the multiplication in the tangent prolongation. However, since $f_0$
        takes its image only in tangent spaces over units in $G$, the formula
        simplifies (cf.\ \cite[Theorem 1.4.14, eq.\
        (4)]{Mackenzie05General-theory-of-Lie-groupoids-and-Lie-algebroids}) to
        \begin{align*}
         Tm(f_0 (t,\beta(g),\eta),0 (g)) &= T(R_g) (f_0 (t,\beta(g),\eta)) - T(1)T\alpha (f_0 (t,\beta (g),\eta)) + 0(g) \\
         &= T(R_g) (\eta^\wedge (t,\beta (g))) - \underbrace{T(1)T\alpha (f_0 (t,\beta (g),\eta))}_{=0(1(\beta (g))}) = f(t,g,\eta).
        \end{align*}
  \item[b)] In a) we have seen that $f$ is a mapping of class $C^{0,\infty}$
        such that for all $g \in G$ we have $f(\cdot , g , \cdot) \in T_g G$.
        Now $G$ is a smooth Banach-manifold by assumption and the map $f$
        satisfies the assumptions of \cite[5.12]{alas2012}. Hence
        \cite[5.12]{alas2012} yields for all choices
        $(t_0, g_0,\eta) \in [0,1] \times G \times C^0([0,1],\Gamma (\Lf(\cG)))$
        a unique maximal solution
        $\varphi_{t_0,g_0,\eta} \colon J_{t_0,g_0,\eta} \rightarrow G$ of
        \eqref{eq:diffeq} defined on some (relatively) open interval
        $t_0 \in J_{t_0,g_0,\eta} \subseteq [0,1]$. We claim that it is
        possible to construct an zero-neighbourhood
        $\Omega \opn \Gamma (\Lf(\cG))$ such that for all
        $(t_0,g_0,\eta) \in [0,1] \times G \times C^0([0,1],\Omega)$ the
        maximal solution $\varphi_{t_0,g_0,\eta}$ is defined on $[0,1]$. If
        this is true, then the flow map $\Fl^{f}$ is defined on
        $[0,1] \times [0,1] \times (G \times  C^0([0,1],\Omega))$. \bigskip
        
        \textbf{Construction of $\Omega$:} We construct the neighbourhood
        $\Omega$ via a local argument in charts. Let us thus fix the following
        symbols for the rest of this proof:
        \begin{itemize}
         \item $F$ is the Banach-space on which $G$ is modelled,
         \item $E$ denotes the complemented subspace of $F$ on which
               $T^\alpha G$ is modelled (cf.\ \ref{setup: alpha:sbd}).
        \end{itemize}
        \textit{Step 1: Reduction to initial values in $M$.} First recall that
        for $t \in J_{t_0,g_0,\eta}$ the following holds
        \begin{equation}\label{eq: Icurv:inv:VF}
         \frac{\partial}{\partial t}\varphi_{t_0,g_0,\eta} (t) = f(t, \varphi_{t_0,g_0,\eta} (t) , \eta) = TR_{\varphi_{t_0,g_0,\eta} (t)} \eta^\wedge (t, \beta (\varphi_{t_0,g_0,\eta} (t))) 
         \stackrel{\eqref{eq: RI:VF}}{=} \overrightarrow{\eta (t)}(\varphi_{t_0,g_0,\eta} (t)).
        \end{equation}
        We conclude that $\varphi_{t_0,g_0,\eta}$ is an integral curve for the
        time-dependent right-invariant vector field\\
        $\overrightarrow{\eta} \colon [0,1] \rightarrow \SectRI{G} , t \mapsto \overrightarrow{\eta (t)}$.
        Arguing as in \cite[p. 132
        3.6]{Mackenzie05General-theory-of-Lie-groupoids-and-Lie-algebroids} we
        derive the following information on the integral curves: Recall that
        $\overrightarrow{\eta (t)}$ is $\alpha$-vertical for each
        $t \in [0,1]$. Thus \eqref{eq: Icurv:inv:VF} yields for $g_0 \in G$ the
        equation $\alpha \circ \varphi_{t_0,g_0,\eta} = \alpha (g_0)$. In
        particular, the integral curve through $1_{\beta (g_0)}$ restricts to a
        mapping
        $J_{t_0,1_{\beta (g_0)},\eta} \rightarrow \alpha^{-1} (\beta (g_0))$.
        Thus
        $c \colon J_{t_0, 1_{\beta (g_0)},\eta} \rightarrow G, t \mapsto R_{g_0} \circ \varphi_{t_0,1_{\beta (g_0)},\eta} (t)$
        is defined. A quick computation shows that $c(t_0) = g_0$ and by
        \eqref{eq: Icurv:inv:VF} we derive
        $\frac{\partial}{\partial t} c(t) = \overrightarrow{\eta (t)}(c(t))$.
        This proves that $c(t)$ is the solution of \eqref{eq:diffeq} passing
        through $g_0$ at time $0$. We conclude that it suffices to construct a
        zero-neighbourhood $\Omega$ such that the solutions
        $\varphi_{t_0,1_{x},\eta}$ are defined on $[0,1]$ for all
        $(x,\eta) \in M \times C^0([0,1],\Omega)$.\bigskip 
        
        \textit{Step 2:
        Integral curves on $[0,1]$ for all initial values in a neighbourhood of
        $x_0 \in M$.} Fix $x_0 \in M$. We choose a submersion chart
        $\kappa_{x_0} \colon U_{x_0} \rightarrow V_{x_0} \subseteq F$ for
        $\alpha$ whose domain contains $1_{x_0}$. Then the tangent chart
        $T\kappa_{x_0}$ is a submersion chart for $T\alpha$, whence it
        restricts to the bundle $\pi_\alpha \colon T^\alpha G \to G$ \ and
        yields a bundle trivialisation
        $T^\alpha \kappa_{x_0} \colon \pi_{\alpha}^{-1} (U_{x_0}) \rightarrow V_{x_0} \times E$
        defined via
        $T^\alpha \kappa_{x_0} = T\kappa_{x_0}|_{\pi_{\alpha}^{-1} (U_{x_0})}^{V_{x_0} \times E}$.
        We remark that $T^\alpha \kappa_{x_0}$ induces a trivialisation of the
        pullback bundle $\Lf(G)\to M$. In the following we will identify
        $\Lf(G)\to M$ with the restriction of $T^\alpha G$ to $1(M)$ and the
        fibres $\Lf(G)_x$ with $T^\alpha_{1_x} G$. Under this identification
        the pullback trivialisation is just the restriction of
        $T^\alpha \kappa_{x_0}$ to $\pi_\alpha^{-1} (1(M))$.
        
        Note that the map $1\circ \beta$ fixes units. 
        Hence, by replacing $U_{x_0}$ with the
        open set $U_{x_0} \cap (1\circ \beta)^{-1} (U_{x_0}) \ni 1_{x_0}$ for
        each $g \in U_{x_0}$ the unit $1_{\beta (g)}$ is also contained in
        $U_{x_0}$. Denote the (smooth) inclusion of $E$ into $F$ by $I_E^F$. We
        define the map
        \begin{equation}\label{eq: hx0}
         h_{x_0} \colon [0,1] \times (V_{x_0} \times C^0 ([0,1] , \Gamma (\Lf(\cG)))_{\text{c.o.}}) \rightarrow F , (t,x,\eta) \mapsto I_E^F \circ \op{pr}_2 \circ T^\alpha \kappa_{x_0} \circ f  (t , \kappa_{x_0}^{-1} (x),\eta)
        \end{equation}
        which is of class $C^{0,\infty}$ by the chain rules \cite[Lemma 3.17
        and Lemma 3.18 ]{alas2012} and part (a). For later use we record that
        for all $X \in \pi_\alpha^{-1} (U_{x_0}) \subseteq TG$ we have
        $I_E^F \circ \op{pr}_2 \circ T^\alpha \kappa_{x_0} (X) = \pr_2 \circ T\kappa_{x_0} (X)$.
        Note that since $1\circ \beta (U_{x_0}) \subseteq U_{x_0}$, we can
        rewrite $h_{x_0} (t,y , \eta)$ for fixed
        $(t,y) \in [0,1] \times V_{x_0}$ as
        \begin{align*}
         h_{x_0} (t,y , \eta )  &= \pr_2 \circ T\kappa_{x_0}  \circ f  (t , \kappa_{x_0}^{-1} (y),\eta) 
         = \pr_2 \circ T\kappa_{x_0} \circ TR_{\kappa_{x_0}^{-1} (y)} (\underbrace{\eta^\wedge (t,\beta (\kappa_{x_0}^{-1} (y)))}_{\in \pi_\alpha^{-1} (U_{x_0})}) \\
         &= \pr_2 \circ T\kappa_{x_0}  \circ TR_{\kappa_{x_0}^{-1} (y)} \circ T^\alpha \kappa_{x_0}^{-1} \circ T^\alpha \kappa_{x_0} (\eta^\wedge (t,\beta (\kappa_{x_0}^{-1} (y)))) \\ 
         &= \pr_2 \circ T\kappa_{x_0}  \circ TR_{\kappa_{x_0}^{-1} (y)} \circ T^\alpha \kappa_{x_0}^{-1} (\kappa_{x_0} (1_{\beta (\kappa_{x_0}^{-1} (y))}) , \pr_2 T^\alpha \kappa_{x_0} (\eta^\wedge (t,\beta (\kappa_{x_0}^{-1} (y))))).
        \end{align*}
        Hence we obtain for each $y \in V_{x_0}$ a continuous linear map
        \begin{displaymath}
         l_{x_0,y} \colon E \rightarrow F, l_{x_0,y} (\omega) = \pr_2 \circ
         T\kappa_{x_0} \circ TR_{\kappa_{x_0}^{-1} (y)} \circ T^\alpha
         \kappa_{x_0}^{-1} (\kappa_{x_0} (1_{\beta (\kappa^{-1}_{x_0} (y))}) ,
         \omega).
        \end{displaymath}
        Set $z_0 \coloneq \kappa_{x_0} (1_{x_0})$. By Lemma \ref{lemma:
        cont:dep} the assignment
        $l_{x_0} \colon V_{x_0} \rightarrow \BoundOp{E,F} , y \mapsto l_{x_0,y}$
        is continuous. Hence we obtain an open $z_0$-neighbourhood
        $W_{z_0} \subseteq V_{x_0}$ such that
        $\sup_{w \in W_{z_0}} \opnorm{l_{x_0, w}} \leq B_{x_0} \coloneq \opnorm{l_{x_0,z_0}} +1$.
        Now $\beta$ is a submersion, whence an open map. We conclude from
        $\beta \kappa_{x_0}^{-1} (z_0) = x_0$ that
        $\beta \circ \kappa_{x_0}^{-1}(V_{x_0})$ is an open neighbourhood of
        $x_0$. Choose a compact $x_0$-neighbourhood
        $A_{x_0} \subseteq \beta (U_{x_0})$ and note that
        $z_0 \in (\beta \circ \kappa_{x_0}^{-1})^{-1} (A_{x_0}^\circ) \opn V_{x_0}$.
        Hence there is $R_{x_0}> 0$ with
        $\overline{B_{2R_{x_0}} (z_0)} \subseteq W_{z_0} \cap (\beta \circ \kappa_{x_0}^{-1})^{-1} (A_{x_0}^\circ))$.
        
        By Lemma \ref{lemma: Lipschitz} we can shrink $R_{x_0}$ and choose a
        zero-neighbourhood $N_{0} \opn \Gamma (\Lf(\cG))$ such that the map
        $h_{x_0}$ is uniformly Lipschitz continuous in the Banach-space
        component on
        $[0,1] \times C^0([0,1], N_0) \times \overline{B_{2R_{x_0}} (z_0)}$.
        Fix $v \in B_{R_{x_0}} (z_0)$ and estimate the supremum of the norm of
        $h_{x_0}$ on
        $[0,1] \times C^0([0,1] , N_0) \times \overline{B_{R_{x_0}} (v)}$. By
        choice of $v$ the ball $\overline{B_{R_{x_0}} (v)}$ is contained in
        $B_{2R_{x_0}} (z_0) \subseteq W_{x_0}$. Thus for all
        $y \in \overline{B_{R_{x_0}} (v)}$ we have the upper bound $B_{x_0}$
        for $\opnorm{l_{x_0,y}}$. We obtain an estimate the supremum of
        $\norm{h_{x_0} (t,y,\eta)}_F$ over
        $(t,y,\eta) \in [0,1] \times C^0([0,1] , N_0) \times \overline{B_{R_{x_0}} (v)}$
        as
        \begin{equation}\label{eq: est:supnorm}
         \begin{aligned}
         \sup_{(t,y,\eta)} \norm{h_{x_0} (t,y,\eta)}_F &\leq \sup_{(t,y,\eta)} \opnorm{l_{x_0, y}} \norm{\pr_2 T^\alpha \kappa_{x_0} \circ \eta^\wedge (t,\beta (\kappa_{x_0}^{-1} (y)))}_E\\
         &\leq B_{x_0} \sup_{(t,y,\eta)} \norm{\pr_2 T^\alpha \kappa_{x_0} \circ \eta^\wedge (t,\beta (\kappa_{x_0}^{-1} (y)))}_E.
         \end{aligned}
        \end{equation}
        By construction of $B_{2R_{x_0}} (z_0)$ we have
        $\beta \circ \kappa_{x_0}^{-1} (B_{2R_{x_0}} (z_0)) \subseteq A_{x_0}$
        and $A_{x_0}$ is a compact subset of $1^{-1} (U_{x_0}) \subseteq M$.
        Now $T^\alpha \kappa_{x_0}$ restricts to a trivialisation of the
        pullback bundle $\Lf(G)\to M$ (identified with a subset of
        $T^\alpha G$). Recall that $A_{x_0}$ is a compact set. Hence for each
        open set $O \subseteq E$ by definition of the compact open topology and
        Definition \ref{def:smooth_compact_open_topology} the set
        $\lfloor A_{x_0} , O\rfloor \coloneq \{f \in C^\infty (\beta (U_{x_0}), E)\mid f(A_{x_0}) \subseteq O\}$
        is open in $C^\infty (\beta (U_{x_0}),E)$. Let
        $\res \colon \Gamma (\Lf (\cG)) \rightarrow \Gamma (\left.\Lf (\cG)\right|_{\beta (U_{x_0})}), f \mapsto f|_{\beta (U_{x_0})}$.
        Then a combination of \cite[Lemma
        5.5]{Wockel13Infinite-dimensional-and-higher-structures-in-differential-geometry}
        with Theorem \ref{thm: sect} implies that
        \begin{displaymath}
         D \coloneq \left\{X \in \Gamma (\Lf(\cG)) \middle| \sup_{y \in
         A_{x_0}} \norm{\pr_2 \circ T^\alpha \kappa_{x_0} \circ X (y)}_E <
         \frac{R_{x_0}}{B_{x_0}}\right\} = ((\pr_2 \circ T^\alpha
         \kappa_{x_0})_* \circ \res)^{-1}\left(B_{\frac{R_{x_0}}{B_{x_0}}}
         (0)\right)
        \end{displaymath}
        is an open neighbourhood of the zero-section in $\Gamma (\Lf(\cG))$.
        
        Define $\Omega_0 \coloneq D \cap N_0$. Clearly
        $\Omega_0 \opn \Gamma (\Lf(\cG))$ is a zero-neighbourhood and
        $C^0([0,1], \Omega_0)$ is open in
        $C^0([0,1], \Gamma (\Lf(\cG)))_{\text{c.o.}}$. From \eqref{eq:
        est:supnorm} we derive that for all $v \in B_{R_{x_0}} (z_0)$ and
        $\eta \in C^0([0,1], \Omega_{x_0})$ the estimate:
        \begin{equation}\label{eq: final:estimate}
         \sup_{(t,y,\eta) \in [0,1] \times B_{R_{x_0}}(v) \times  C^0([0,1], \Omega_{x_0})} \norm{h_{x_0} (t,y,\eta)}_F < R_{x_0}
        \end{equation}
        We will now solve the initial value problem \eqref{eq:diffeq} for fixed
        $\eta \in  C^0([0,1], \Omega_{x_0})$ and
        $(t_0,v) \in [0,1] \times B_{R_{x_0}} (z_0)$:
        \begin{equation}\label{eq: ODE:loc}
         \begin{cases}
         \frac{\partial}{\partial t}c(t) = h_{x_0} (t,c(t),\eta), \\
         c(t_0) = v.
         \end{cases}
        \end{equation}
        From the argument above, we know that
        $h_{x_0} (\cdot, \eta) \colon [0,1] \times B_{2R_{x_0}} (z_0) \rightarrow F$
        satisfies a uniform Lipschitz condition in the Banach space component
        (i.e.\ in $ B_{2R_{x_0}} (z_0)$). In addition, \eqref{eq:
        final:estimate} shows that for all $\eta \in  C^0([0,1], \Omega_{x_0})$
        we obtain
        $M \coloneq \sup_{(t, x) \in [0,1] \times \overline{B_{R_{x_0} (v)}}} \norm{h_{x_0} (t,x,\eta)}_E < R_{x_0}$.
        Observe that $\frac{R_{x_0}}{M} \geq 1$. Now a combination of \cite[7.4
        Local Existence and Uniqueness Theorem]{amann1990} with \cite[Remark
        7.10 (a)]{amann1990} shows that the initial value problem \eqref{eq:
        ODE:loc} admits a unique solution
        $c_{t_0,v,\eta}\colon [0,1] \rightarrow B_{R_{x_0}} (v)$.
        \footnote{Note that the proof of \cite[7.4]{amann1990} for
        infinite-dimensional Banach spaces requires a uniform Lipschitz
        condition on all of $\overline{B_{2R_{x_0}}(z_0)}$ (cf.\ \cite[Remark
        7.5 (b)]{amann1990}).} It is easy to see that
        $\kappa_{x_0}^{-1} \circ c_{t_0,v,\eta}$ is just the integral curve
        $\varphi_{t_0,\kappa_{x_0}^{-1}(v), \eta}$, whence this curve must
        exist on $[0,1]$. Hence for all
        $(t_0, y,\eta) \in [0,1] \times \kappa_{x_0} (B_{R_{x_0}} (z_0)) \times  C^0([0,1], \Omega_{x_0})$
        the initial value problem \eqref{eq:diffeq} admits a unique solution
        $\varphi_{t_0,y,\eta}$ on $[0,1]$. Finally we remark that
        $\kappa_{x_0} (B_{R_{x_0}} (z_0))$ is an open neighbourhood of
        $1_{x_0}$. \bigskip
        
        \textit{Step 3: Define $\Omega$.} We construct for
        each $x \in M$ as in Step 2 an open neighbourhood $\cN_x \subseteq G$
        of $x$ and a zero-neighbourhood $\Omega_x \opn \Gamma (\Lf(\cG))$. By
        construction the solution $\varphi_{t_0,y,\eta}$ of \eqref{eq:diffeq}
        exists on $[0,1]$ for all
        $(t_0,y,\eta) \in [0,1] \times \cN_x \times C^0 ([0,1],\Omega_x)$.
        Since $M$ is compact, there is a finite set
        $x_1, \ldots, x_n \in M , n \in \N$ such that
        $G_{x_0} \subseteq \bigcup_{1\leq i\leq n} \cN_{x_i}$. Then
        $\Omega \coloneq \bigcap_{1\leq i \leq n} \Omega_{x_i}$ is an open
        zero-neighbourhood in $\Gamma (\Lf(\cG)))$. By construction for all
        $(t_0,x,\eta) \in [0,1] \times M \times C^0([0,1],\Omega)$ the solution
        $\varphi_{t_0,1_x,\eta}$ of \eqref{eq:diffeq} exists on $[0,1]$.
  \item[c)] Let $r \in \N_0 \cup \{ \infty \}$. For
        $\eta \in C^0 ([0,1] , \Omega)$ we derive from (b) that the integral
        curves for \eqref{eq:diffeq} exist on $[0,1]$. From (a) we know that
        $f$ is of class $C^{0,\infty}$. Hence \cite[Proposition 5.13]{alas2012}
        implies that for fixed $t_0 \in [0,1]$ the map
        $\Fl^f (t_0, \cdot) \colon [0,1] \times (G \times C^0 ([0,1] , \Gamma (\Lf(\cG)))_{\text{c.o.}}) \rightarrow G$
        is a mapping of class $C^{1,\infty}$. Specialising to $t_0 = 0$ we see
        that $\Fl^f_0$ is of class $C^{1,\infty}$.
  \item[d)] Fix $\eta \in C^0([0,1],\Gamma (\Lf(\cG)))$ and $s,t \in [0,1]$.
        We have to prove that
        $\Psi_{s,t} \coloneq \beta \circ \Fl^f (s,t , \cdot , \eta) \circ 1 \colon M \rightarrow M$
        is a diffeomorphism. The map $\Psi_{s,t}$ is smooth as a composition of
        smooth mappings. We claim that the inverse of $\Psi_{s,t}$ is given by
        $\Psi_{t,s} \coloneq \beta \circ \Fl^f (t,s,\cdot , \eta) \circ 1$. To
        see this recall the following properties of the flow for
        $t_0,t_1,t_2 \in [0,1]$
        \begin{align} 
         \Fl^f (t_1,t_2 , \cdot ,\eta) \circ \Fl^f (t_0,t_1, \cdot , \eta) = \Fl^f (t_0,t_2, \cdot , \eta) \quad \text{ and } \quad  \Fl^f (t_0,t_0, \cdot , \eta) = \id_{G} (\cdot) \label{eq: flow1}\\
         \forall (h,g) \in G \times_{\alpha ,\beta} G, \text{ by (b) and (c) Step 1: }  R_g \circ \Fl^f (t_0,t_1, h,\eta) = \Fl^f (t_0, t_1, \cdot , \eta) \circ R_g (h) \label{eq: flow2}
        \end{align}
        Furthermore, we observe for $g\in G$ that $R_g^{-1} = R_{g^{-1}}$.
        Hence for $x \in M$ a combination of \eqref{eq: flow1} and \eqref{eq:
        flow2} yields
        $R^{-1}_{\Fl^f (t,s,1_x,\eta)} (1_x) = \Fl^f (s,t, \cdot ,\eta) \circ R^{-1}_{\Fl^f (t,s,1_x,\eta)} \circ \Fl^f (t,s,\cdot,\eta) (1_x) =  \Fl^f (s,t, \cdot ,\eta) (1_{\beta (\Fl^f (t,s,1_x,\eta))})$.
        Together with $\beta (R_g (h)) = \beta (h)$ for all
        $(h,g) \in G \times_{\alpha,\beta} G$, the last observation enables the
        following computation:
        \begin{align*}
         x = \beta (1_x) &= \beta \circ R^{-1}_{\Fl^f (t,s,1_x,\eta)} (1_x) = \beta \circ \Fl^f (s,t, \cdot ,\eta) \circ R^{-1}_{\Fl^f (t,s,1_x,\eta)} \circ \Fl^f (t,s,\cdot ,\eta) (1_x)\\
         &= \beta \circ \Fl^f (s,t, \cdot ,\eta) \circ 1 \circ \beta \circ \Fl^f (t,s,\cdot,\eta) (1_x) = \Psi_{s,t} \circ \Psi_{t,s} (x)
        \end{align*}
        Interchanging the roles of $\Psi_{s,t}$ and $\Psi_{t,s}$ we see that
        indeed $\Psi_{s,t}^{-1} = \Psi_{t,s}$ and $\Psi_{s,t}$ is a
        diffeomorphism.
  \item[e)] The map $1 \colon M \rightarrow G$ is smooth. Thus the chain rule
        \cite[Lemma 3.17]{alas2012} and (c) show that\\
        $H_\eta = \Fl^f_0 (\cdot,\eta) \circ \id_{[0,1]} \times 1$ is a
        $C^{1,\infty}$-map. From \eqref{eq: Icurv:inv:VF} and step 1 in (a) we
        infer $\alpha \circ c_{\eta} (t) = \id_{M}$. Then a combination of
        \cite[1.7]{hg2012} and (d) shows that $c_\eta (t) \in \Bis (\cG)$ and
        thus $c_\eta \colon [0,1] \rightarrow \Bis (\cG)$ makes sense.
        
        To see that $c_\eta$ is continuous, recall from Definition
        \ref{def:smooth_compact_open_topology} that the topology on
        $C^\infty (M,G)$ is initial with respect to the family
        $(T^n \colon C^\infty (M,G) \rightarrow C^0(T^nM,T^nG)_{\text{c.o.}} , f \mapsto T^nf)_{n\in \N_0}$.
        Now $\Bis (\cG)$ is an embedded submanifold whence the topology on
        $\Bis (\cG)$ is the subspace topology induced by $C^\infty (M,G)$.
        Clearly $c_\eta$ will be continuous if for all $n \in \N_0$ the map
        $T^n \circ c_\eta$ is continuous. Observe that $c_\eta^\wedge = H_\eta$
        and the mapping $H_\eta \colon [0,1] \times M \rightarrow G$ is of
        class $C^{1,\infty}$. For $\xi \in T^nM$ and $t \in [0,1]$ fixed, we
        obtain the formula $T^nc_\eta (t) (\xi) = T^n H_\eta (t,\cdot) (\xi)$
        (where we compute the tangent only with respect to the argument in
        $M$). We compute locally to exploit the $C^{1,\infty}$-property of
        $H_\eta$: Choose an open $t$-neighbourhood $U_t \opn [0,1]$ with
        inclusion $i_t \colon U_t \hookrightarrow [0,1]$ and charts $\kappa$ of
        $M$ and $\lambda$ of $G$ such that
        $\lambda \circ H_\eta \circ (i_t \times \kappa)$ is defined. Then we
        see (cf.\ \cite[Lemma
        5.3]{Wockel13Infinite-dimensional-and-higher-structures-in-differential-geometry})
        that
        \begin{equation}\label{eq: Tnloc}
         T^n\lambda \circ T^n H_\eta (t,\cdot) \circ T^n \kappa =T^n(\lambda \circ H_\eta (t,\cdot) \circ \kappa) = T^{n-1} (\lambda \circ H_\eta (t,\cdot) \circ \kappa) \times (d T^{n-1}(\lambda \circ H_\eta (t,\cdot) \circ \kappa)).
        \end{equation}
        Denote by $\R^k$ the model space of $M$ and let
        $f \colon \R^k \supseteq U \rightarrow F$ be a smooth mapping from an
        open subset into a locally convex space. Recall from \cite[p.
        49]{hg2002a} the following variant of the usual differential for $f$:
        Set $d^0f=f$, $d^1f=df$ and
        $d^nf = d(d^{n-1}f) \colon U \times (\R^k)^{2^n-1} \rightarrow F$. Note
        that by \cite[Lemma 1.14]{hg2002a} these differentials exist for any
        $C^{\infty}$-map $f$. Thus \eqref{eq: Tnloc} shows that we can
        recursively split the tangent into a product of derivatives
        $d^r (\lambda \circ H_\eta (t,\cdot) \circ \kappa)$ composed with
        projections. Furthermore, the formula in the proof of \cite[Lemma
        1.14]{hg2002a} and the $C^{1,\infty}$-property of $H_\eta$ show that
        $d^r (\lambda \circ H_\eta (t,\cdot) \circ \kappa_2)$ depends
        continuously on $t$. In conclusion,
        $T^{0,n}H_\eta \colon [0,1] \times T^n M \rightarrow T^n G, (t,\xi) \mapsto T^n H_\eta (t,\cdot) (\xi)$
        is continuous. The manifold $M$ is finite dimensional and so is $T^nM$
        for all $n \in \N_0$. In particular, $T^n M$ is locally compact for
        $n\in \N_0$ and we derive from \cite[Theorem 3.4.1]{Engelking1989} that
        $T^nc_\eta = (T^{0,n}H_\eta)^\vee$ is continuous.
        
        We will prove now that $c_\eta$ is of class $C^1$. 
        It suffices to prove that $c_\eta$ is
        locally of class $C^1$. To do so fix $s \in [0,1]$ and recall some
        facts from the construction of $\Omega$ in (b). The set $\Omega$ was
        constructed with respect to a finite family
        $\cN_i \opn G, 1 \leq i \leq n$ which satisfies:
        \begin{itemize}
         \item[(i)] $1(M) \subseteq \bigcup_{1\leq i\leq n} \cN_i$,
         \item[(ii)] For each $1 \leq i \leq n$ there is a manifold chart
               $\kappa_i \colon U_i \rightarrow V_i \subseteq F$, such that
               $\Fl^f (0, \cdot )|_{[0,1] \times \cN_i \times C^0 ([0,1] , \Omega)}$
               takes its values in $U_i$ (cf.\ (b) Step 2). Moreover,
               $\kappa_i$ is a submersion chart for $\alpha$, whence \\
               $T^\alpha \kappa_i \coloneq T\kappa_i|_{T^\alpha U_i}^{E} \colon T^\alpha U_i \rightarrow U_i \times E$
               is a trivialisation of $T^\alpha G$.
        \end{itemize}
        Set $g \coloneq c_\eta (s) \in \Bis (\cG)$. Recall from Proposition
        \ref{Proposition: SectMFD} the form of a manifold chart around $g$ of
        $\Bis (\cG)$:
        $\varphi_{g} \colon O_{g} \rightarrow \varphi_g (O_g) \opn \Gamma (g^* T^\alpha G)$,
        where
        $O_g \coloneq \{s \in \Bis (\cG) \mid \forall x \in M,\ (s(x),g(x)) \in Q\}$
        for a a fixed neighbourhood $Q$ of the diagonal in $N\times N$.
        
        Consider the $C^{1,\infty}$-mapping
        $c_\eta^\wedge \coloneq H_\eta \colon [0,1] \times M \rightarrow G$.
        Then
        $(g \circ \pr_2 , c_\gamma^\wedge) \colon [0,1] \times M \rightarrow G \times G$
        is continuous with
        $\{s\} \times M \subseteq (g \circ \pr_2 , c_\gamma^\wedge)^{-1} (V) \opn [0,1] \times M$.
        Hence there is a relatively open interval $s \in J_s \subseteq [0,1]$
        such that $c_\eta (\overline{J_s}) \subseteq O_g$. We will prove that
        $\varphi_g \circ c_\eta|_{J_s} \colon J_s \rightarrow \Gamma (g^* T^\alpha G)$
        is $C^1$. Note first that $c_\eta (t)$ maps $1^{-1} (\cN_i) \opn M$
        into the chart domain $U_i$ for all $t \in J_s$. In particular, for
        $z \in 1^{-1} (\cN_i)$ the compact set
        $c_\eta^\wedge (\overline{J_s} \times \{z\}) \times \{g(z)\}$ is
        contained in $V\cap (U_i \times U_i)$. We apply Wallace Theorem
        \cite[3.2.10]{Engelking1989} to obtain open neighbourhoods
        $U_{z,s}, V_z \opn G$ with
        $c_\eta^\wedge (\overline{J_s} \times \{z\}) \times \{g(z)\} \subseteq U_{z,s} \times V_z \opn V\cap (U_i \times U_i)$.
        The set $(c_\eta^\wedge)^{-1} (U_{z,s})$ is an open neighbourhood of
        $\overline{J_s} \times \{z\} \in [0,1] \times 1^{-1} (\cN_i)$. Apply
        Wallace Theorem again to find an open $z$-neighbourhood
        $W_z \opn 1^{-1} (\cN_i)$ such that $g(W_z) \subseteq V_z$ and
        $J_s \times W_z \subseteq (c_\eta^\wedge)^{-1} (U_{z,s})$. By (i) the
        open sets $(1^{-1} (\cN_i))_{1\leq i \leq n}$ cover $M$. Thus we repeat
        the construction of $W_z$ for all $z \in M$. By compactness of $M$,
        there are finitely many $z_j , 1\leq j\leq m$ such that
        $M = \bigcup_{1 \leq j \leq m}W_{z_j}$. For each $1 \leq j \leq m$ we
        choose $\kappa_{i_j}$ such that
        $c_\eta (J_s \times W_{z_j}) \times g(W_{z_j}) \subseteq U_{z_j,s} \times V_{z_j} \subseteq U_{i_j} \times U_{i_j}$.
        Note that by (ii) the trivialisations $T^\alpha \kappa_{i_j}$ of
        $T^\alpha G$ induce an atlas of trivialisations for the bundle
        $g^* T^\alpha G$. From \cite[Proposition
        7.3]{Wockel13Infinite-dimensional-and-higher-structures-in-differential-geometry}
        we recall that the topology of $\Gamma (g^* T^\alpha G)$ is initial
        with respect to
        \begin{displaymath}
         \Phi \colon \Gamma (g^* T^\alpha G) \rightarrow \prod_{1\leq j \leq n}
         C^\infty (W_{z_j}, E) , \sigma \mapsto \pr_2 \circ T^\alpha
         \kappa_{i_j} \circ \sigma|_{W_{z_j}}
        \end{displaymath}
        and that $\Phi$ is a linear topological embedding with closed image.
        Thus $\varphi_g \circ c_\eta|_{J_s}$ will be of class $C^1$ if and only
        if for each $1 \leq j \leq m$ the map
        $p_j \circ \Phi \circ \varphi_g \circ c_\eta|_{J_s} \colon J_s \rightarrow C^\infty (W_{z_j}, E)$
        is $C^1$. Here $p_j$ is the projection onto the $j$-th component of the
        product. The manifold $W_{z_j} \opn M$ is finite dimensional and
        $J_s \opn [0,1]$ is a locally convex subset with dense interior of
        $\R$. We can thus apply the exponential law for $C^{1,\infty}$-maps
        \cite[Theorem 4.6 (d)]{alas2012}: The map
        $\varphi_g \circ c_\eta|_{J_s}$ will be of class $C^1$ if and only if
        for each $1\leq j \leq m$ the map
        $(\pr_j \circ \Phi \circ \varphi_g \circ c_\eta|_{J_s})^\wedge \colon J_s \times W_{z_j} \rightarrow E$
        is of class $C^{1,\infty}$. Let
        $\pi_\alpha \colon T^\alpha G \rightarrow G$ be the bundle projection
        and denote by $\A$ the local addition on $G$ adapted to $\alpha$. Using
        the description of the chart $\varphi_g$ we compute an explicit formula
        for $(\pr_j \circ \Phi \circ \varphi_g \circ c_\eta|_{J_s})^\wedge$:
        \begin{equation}\label{eq: flow:loc} \begin{aligned}
         (\pr_j \circ \Phi \circ \varphi_g \circ c_\eta|_{J_s})^\wedge (t,x) &=  (\pr_2 T^\alpha \kappa_{i_j} \circ (\pi_\alpha, \A)^{-1} \circ (g,c_\eta|_{J_s}))^\wedge (t,x)\\
         &=  (\pr_2 T^\alpha \kappa_{i_j} \circ (\pi_\alpha, \A)^{-1} \circ (g,c_\eta (t)) (x)) \\
         &=  \pr_2 T^\alpha \kappa_{i_j} \circ  (\pi_\alpha, \A)^{-1} \circ (g  \circ \pr_{W_{z_j}} ,c_\eta^\wedge|_{J_s\times W_{z_j}}) (t,x)   
         \end{aligned}
        \end{equation}
        Here $\pr_{W_{z_j}} \colon J_s \times W_{z_j} \rightarrow W_{z_j}$ is
        the canonical (smooth) projection. By construction of the open sets
        $W_{z_j}$, the smooth mapping
        $\pr_2 T^\alpha \kappa_i \circ  (\pi_\alpha, \A)^{-1} \circ (\id_{U_{z,s}}, g)$
        is defined on the product $U_{z_j , s} \times W_{z_j}$. Furthermore
        $c_\eta^\wedge|_{J_s\times W_{z_j}} = H_\eta|_{J_s \times W_{z_j}}$
        holds and thus $c_\eta^\wedge|_{J_s\times W_{z_j}}$ is of class
        $C^{1,\infty}$. We have
        $c_\eta^\wedge (J_s \times W_{z_j}) \subseteq U_{z_j,s}$. Computing in
        local charts, the chain rule \cite[Lemma 3.19]{alas2012} together with
        \eqref{eq: flow:loc} implies that\\
        $(\pr_j \circ \Phi \circ \varphi_g \circ c_\eta|_{J_s})^\wedge$ is a
        mapping of class $C^{1,\infty}$. We conclude that
        $c_\eta|_{J_s} \colon J_s \rightarrow \Bis (\cG)$ is of class $C^1$,
        whence the assertion follows.
 \end{enumerate}
\end{proof}

\begin{lemma}\label{lemma:
 cont:dep} In the situation of Proposition \ref{proposition: flow}
 \ref{proposition: flow_b} denote by $F$ the model space of $G$ and by $E$ the
 typical fibre of $T^\alpha G$. Let $\kappa \colon U \rightarrow V \subseteq F$
 be a submersion chart for $\alpha$ such that for all $g \in U$ we have
 $1_{\beta (g)} \in U$. Furthermore, we denote by $T^\alpha \kappa$ the
 trivialisation of the bundle $T^\alpha G$ obtained by restriction of $T\kappa$
 to $T^\alpha G$. Then the map
 \begin{displaymath}
  l \colon V \rightarrow \BoundOp{E,F} , y \mapsto l_{y} \quad \text{ with } l_{y}(\omega) = \pr_2 \circ T\kappa \circ TR_{\kappa^{-1} (y)} \circ T^\alpha \kappa^{-1} (\kappa (1_{\beta (\kappa^{-1}(y))}) , \omega)
 \end{displaymath}
 is continuous with respect to the operator-norm topology on $\BoundOp{E,F}$.
\end{lemma}

\begin{proof}
 Before we tackle the continuity, we begin with some preliminaries: For each
 $y\in V$, the element $(1_{\beta (\kappa (y))} , \kappa^{-1} (y))$ is
 contained in the domain $G \times_{\alpha ,\beta} G$ of the groupoid
 multiplication $m$. Define the map
 $u \colon V \rightarrow G \times_{\alpha , \beta} G , v \mapsto (1_{\beta (\kappa^{-1} (v))}, \kappa^{-1} (v))$.
 As $G \times_{\alpha ,\beta} G$ is a split submanifold of $G \times G$
 (modelled on a complemented subspace $H$ of $F\times F$), the map $u$ is
 smooth as a mapping into the submanifold. Note that for all $g \in U$ the unit
 $1_{\beta (g)}$ is contained in $U$, whence
 $u (V) \subseteq (U\times U) \cap (G \times_{\alpha, \beta}G)$. It suffices to
 check continuity of $l$ locally. To this end fix $v \in V$. Let
 $\tau \colon U_\tau \rightarrow V_\tau \subseteq F\times F$ be a submanifold
 chart for $G \times_{\alpha ,\beta} G$ around $u(v)$, i.e.\
 $\tau (U_\tau \cap (G \times_{\alpha, \beta} G)) = V_\tau \cap H$ and
 $u(v) \in U_\tau$. Shrinking $U_\tau$ we can assume that
 $m (U_\tau \cap (G \times_{\alpha,\beta} G)) \subseteq U$ holds. We obtain an
 open $v$-neighbourhood $W_v \coloneq u^{-1} (U_\tau) \subseteq V \subseteq F$
 such that
 $u(W_v) \subseteq U_\tau \cap \left((U\times U) \times (G \times_{\alpha,\beta}G)\right)$.
 Now back to $l$: As explained in Proposition \ref{proposition: flow}
 \ref{proposition: flow_a}, we can rewrite the formula for $l$ for all
 $y \in W_v$ as follows
 \begin{equation}\label{eq: rewrite} 
  \begin{aligned}
  l (y) &= \pr_2 \circ T\kappa \circ Tm (T^\alpha \kappa^{-1} (\kappa (1_{\beta (\kappa^{-1}(y))}) , \omega) , 0(\kappa^{-1} (y))) \\ 
  &= \pr_2 \circ T\kappa \circ Tm (T^\alpha \kappa^{-1} \times T^\alpha \kappa^{-1} ((\kappa (1_{\beta (\kappa^{-1}(y))}),\omega)(y,0)))\\
  &= \pr_2 \circ T(\kappa \circ m \circ \tau^{-1}|_{V_\tau \cap H}) (T\tau \circ (T^\alpha \kappa^{-1} \times T^\alpha \kappa^{-1}) ((\kappa (1_{\beta (\kappa^{-1}(y))}) , \omega)(y,0)))
  \end{aligned}
 \end{equation}
 The above formula shows that the map $l$ splits into several components. We
 exploit this splitting to prove continuity of $l|_{W_v}$. First consider
 $\pr_2 \circ T(\kappa \circ m \circ \tau^{-1}|_{V_\tau \cap H})$. The mapping
 $\kappa \circ m \circ \tau^{-1}|_{V_\tau \cap H} \colon V_\tau \cap H \rightarrow V$
 is well-defined and smooth. We see that
 $\pr_2 \circ T(\kappa \circ m \circ \tau^{-1}|_{V_\tau \cap H}) = d(\kappa \circ m \circ \tau^{-1}|_{V_\tau \cap H}) \colon (V_\tau \cap H) \times H \rightarrow F$.
 Since $H$ and $F$ are Banach spaces, \cite[Lemma 2.10]{Milnor1982} implies
 that
 $p \colon V_\tau \cap H \rightarrow \BoundOp{H,F} , y \mapsto d(\kappa \circ m \circ \tau^{-1}|_{V_\tau \cap H}) (y,\cdot)$
 is continuous.
 
 Now we deal with
 $T\tau \circ (T^\alpha \kappa^{-1} \times T^\alpha \kappa^{-1})$: Recall that
 $\kappa \colon  U \rightarrow V \subseteq F$ is a chart for $G$ and for the
 bundle trivialisation we have
 $T^\alpha \kappa = T\kappa|_{\pi_\alpha^{-1} (U)}^{V\times E}$. Hence
 $T\tau \circ (T^\alpha \kappa^{-1} \times T^\alpha \kappa^{-1})$ is the
 restriction of $T(\tau \circ (\kappa^{-1} \times \kappa^{-1}))$ to the subset
 $((\kappa \times \kappa ((U \times U) \cap U_\tau)) \times E \times E)$. Again
 from \cite[Lemma 2.10]{Milnor1982} it follows that
 \begin{displaymath}
  q \colon \kappa \times \kappa ((U\times U) \cap U_\tau) \rightarrow \BoundOp{F \times F , F \times F}, x \mapsto d (\tau \circ \kappa^{-1} \times \kappa^{-1}|_{\kappa \times \kappa ((U\times U) \cap U_\tau)}) (x,\cdot)
 \end{displaymath}
 is a continuous map. Let $I_E^F$ be the canonical (smooth) inclusion of $E$
 into $F$. Then for all $x$ in the domain of $q$ we derive
 \begin{equation}\label{eq: q}
  q (x) \circ (I_{E}^F \times I_{E}^F) (\cdot)= \pr_2 T \tau \circ (T^\alpha \kappa^{-1} \times T^\alpha \kappa^{-1}) (x,\cdot).
 \end{equation}
 We define the canonical inclusion
 $I_E^{E\times E} \colon E \rightarrow E \times \{0\} \subseteq E \times E$ and
 note that
 $(I_E^F \times I_E^F) \circ I_E^{E\times E} \in \BoundOp{E, F\times F}$. The
 Banach space $H$ is a split subspace of $F \times F$. Hence the projection
 $\pi_H \colon F\times F \rightarrow H$ is continuous. The composition of
 continuous linear maps between Banach spaces is jointly continuous. Thus we
 obtain a continuous map
 \begin{displaymath}
  z \colon \kappa \times \kappa ((U\times U) \cap U_\tau) \rightarrow \BoundOp{E,H} , x\mapsto \pi_H \circ q(x) \circ (I_E^F \times I_E^F) \circ I_E^{E\times E}
 \end{displaymath}
 Now for $x \in (U\times U) \cap U_\tau \cap G \times_{\alpha, \beta} G$ we set
 $(x_1,x_2) \coloneq \kappa \times \kappa (x)$ and let $Y \in E$. Then the
 identity \eqref{eq: q} together with $\tau$ being a submanifold chart for
 $G \times_{\alpha, \beta} G$ shows
 \begin{displaymath}
  q(x) \circ (I_E^F \times I_E^F) \circ I_E^{E\times E} (Y) = \pr_2 T_{x} \tau (\underbrace{T^\alpha_{x_1} \kappa^{-1} (Y) \times 0_{\kappa^{-1} (x_2)}}_{\in TG \times_{T\alpha , T\beta} TG \stackrel{\eqref{eqn1}}{=} T(G \times_{\alpha, \beta} G)})  \in H
 \end{displaymath}
 In particular, $z(x)(Y) = q(x) \circ (I_E^F \times I_E^F) (Y,0)$. With the
 help of \eqref{eq: q} insert the mappings $p$ and $z$ into \eqref{eq: rewrite}
 to derive for $y \in W_v$ the identity
 \begin{displaymath}
  l (y) = p (\tau \circ u(y)) \circ z((\kappa \times \kappa)\circ u(y)) \in \BoundOp{E,F}.
 \end{displaymath}
 We conclude that $l|_{W_v}$ is continuous and the assertion of the Lemma
 follows.
\end{proof}

\begin{lemma}\label{lemma:
 Lipschitz} Let
 $h_{x_0} \colon [0,1] \times V_{x_0} \times C^{0} ([0,1] , \Gamma (\Lf(\cG))) \rightarrow F$
 be the map defined in \eqref{eq: hx0}. Then there is an open
 zero-neighbourhood $N_{0} \opn \Gamma (\Lf(\cG))$ and $\varepsilon > 0$ such
 that $h_{x_0}|_{[0,1] \times C^0([0,1], N_{0}) \times B_\varepsilon (z_0)}$ is
 uniformly Lipschitz continuous with respect to the Banach space component.\\
 Here we define
 $C^0([0,1], N_{0}) = \{\eta \in C^0 ([0,1], \Gamma (\Lf(\cG))) \mid \eta ([0,1])\subseteq N_0\} \opn C^0 ([0,1],  \Gamma (\Lf(\cG)))_{\text{c.o.}}$.
\end{lemma}

\begin{proof}
 Reordering the product on which $h_{x_0}$ is defined, we identify
 $h_{x_0}$ with a $C^{0,\infty}$-mapping with respect to the decomposition
 $([0,1] \times C^{0}([0,1], \Gamma (\Lf(\cG)))_{\text{c.o.}}) \times V_{x_0}$. By
 \cite[Proposition 6.3]{amann1990} the map $h_{x_0}$ is locally Lipschitz
 continuous with respect to the Banach space component. Consider the constant
 map $\textbf{0} \colon [0,1] \rightarrow \Gamma (\Lf(\cG))$ whose image is the
 zero section in $\Gamma (\Lf(\cG))$. Since
 $[0,1] \times \{\textbf{0}\} \times \{z_0\}$ is compact, there are finitely
 many indices $1 \leq i \leq n, n \in \N$ such that
 $[0,1] = \bigcup_{1\leq i \leq n} J_i$ for $J_i \opn [0,1]$ and the following
 holds: 
 
 For each $1 \leq i \leq n$ there is
 $U_i \opn C^{0}([0,1], \Gamma (\Lf(\cG)))$ and $\varepsilon_i >0$ such that on
 $J_i \times U_i \times B_{\varepsilon_i} (z_0)$ the mapping $h_{x_0}$ is
 Lipschitz continuous with respect to $z \in B_{\varepsilon_i} (z_0)$.
 
 We let $\lambda_i$ be the minimal Lipschitz constant for $h_{x_0}$ on
 $J_i \times U_i \times B_{\varepsilon_i} (z_0)$.  Let
 $\cP$ be a subbasis for the topology of $\Gamma (\Lf(\cG))$. Since $[0,1]$ is
 compact, the sets
 $C^0 ([0,1], W)\coloneq \{f \in C^0([0,1],\Gamma(\Lf(\cG))) | f([0,1]) \subseteq W\}$
 with $W \in \cP$ form a subbasis of the compact-open topology on
 $C^0([0,1], \Gamma (\Lf(\cG)))$. Hence there are
 $W_1, \ldots W_m \in \cP, m \in \N$ such that
 $C^0 ([0,1], \bigcap_{1\leq j \leq m} W_j) = \bigcap_{1\leq j \leq m} C^0 ([0,1], W_j) \subseteq \bigcap_{1\leq i\leq n} U_i$
 is a $\textbf{0}$-neighbourhood. Define
 $N_{0} \coloneq \bigcap_{1 \leq j \leq n} W_j$. Since
 $\textbf{0} \in C^0 ([0,1], N_0)$ and thus $\textbf{0} (1) = 0 \in N_0$ holds,
 $N_0 \opn \Gamma (\Lf(\cG))$ is a zero-neighbourhood. Now define
 $L \coloneq \max_{1\leq i \leq n } \{\lambda_i\}$ and
 $\varepsilon = \frac{\min_{1\leq i \leq n} \varepsilon_i}{2}$.
 We consider
 $(t,\eta,x) (t,\eta , y) \in [0,1] \times C^0([0,1], N_{0}) \times B_{\varepsilon} (z_0)$
 such that $t \in J_i$. From
 $\norm{x-y}_E < \varepsilon \leq \frac{\varepsilon_i}{2}$ we derive
 $\norm{y-z_0}_E \leq \norm{y-x}_E + \norm{x-z_0}_E < \varepsilon_i$. Thus
 $\norm{h_{x_0} (t,\eta,x) - h_{x_0} (t,\eta,y)}_F \leq \lambda_i \norm{x-y}_E \leq L \norm{x-y}_E$
 and $L$ is a Lipschitz constant for $h_{x_0}$ on
 $[0,1] \times C^0([0,1], N_{0}) \times B_{\varepsilon} (0)$.
 In conclusion, $h_{x_0}$ is uniformly Lipschitz continuous with respect to the
 Banach space component.
\end{proof}

\appendix

\section{Locally convex manifolds and spaces of smooth maps}\label{Appendix:
MFD}

In this appendix we collect the necessary background on the theory of manifolds
that are modelled on locally convex spaces and how spaces of smooth maps can be
equipped with such a structure. Let us first recall some basic facts concerning
differential calculus in locally convex spaces. We follow
\cite{hg2002a,BertramGlocknerNeeb04Differential-calculus-over-general-base-fields-and-rings}.

\begin{definition}\label{defn:
 deriv} Let $E, F$ be locally convex spaces, $U \subseteq E$ be an open subset,
 $f \colon U \rightarrow F$ a map and $r \in \N_{0} \cup \{\infty\}$. If it
 exists, we define for $(x,h) \in U \times E$ the directional derivative
 $$df(x,h) \coloneq D_h f(x) \coloneq \lim_{t\rightarrow 0} t^{-1} (f(x+th) -f(x)).$$
 We say that $f$ is $C^r$ if the iterated directional derivatives
 \begin{displaymath}
  d^{(k)}f (x,y_1,\ldots , y_k) \coloneq (D_{y_k} D_{y_{k-1}} \cdots D_{y_1}
  f) (x)
 \end{displaymath}
 exist for all $k \in \N_0$ such that $k \leq r$, $x \in U$ and
 $y_1,\ldots , y_k \in E$ and define continuous maps
 $d^{(k)} f \colon U \times E^k \rightarrow F$. If $f$ is $C^\infty$ it is also
 called smooth. We abbreviate $df \coloneq d^{(1)} f$.
 
 From this definition of smooth map there is an associated concept of locally
 convex manifold, i.e., a Hausdorff space that is locally homeomorphic to open
 subsets of locally convex spaces with smooth chart changes. See
 \cite{Wockel13Infinite-dimensional-and-higher-structures-in-differential-geometry,neeb2006,hg2002a}
 for more details.
\end{definition}

\begin{definition}[Differentials
 on non-open sets]\label{defn: nonopen}
 \begin{enumerate}
  \item A subset $U$ of a locally convex space $E$ is called \emph{locally
        convex} if every $x \in U$ has a convex neighbourhood $V$ in $U$.
  \item Let $U\subseteq E$ be a locally convex subset with dense interior and $F$ a locally convex space. 
        A continuous mapping $f \colon U \rightarrow F$ is called $C^r$ if
        $f|_{U^\circ} \colon U^\circ \rightarrow F$ is $C^r$ and each of the
        maps $d^{(k)} (f|_{U^\circ}) \colon U^\circ \times E^k \rightarrow F$
        admits a continuous extension
        $d^{(k)}f \colon U \times E^k \rightarrow F$ (which is then necessarily
        unique). 
        Analogously, we say that a continuous map $g \colon U \rightarrow M$ to a smooth manifold $M$ is of class $C^r$ if the tangent maps 
        $T^{k} (f|_{U^\circ}) \colon U^\circ \times E^{2^k-1} \rightarrow T^kM$ exist and 
        admit a continuous extension $T^{k}f \colon U \times E^{2^k-1} \rightarrow T^kM$. 
        Note that we defined $C^k$-mappings on locally convex sets with dense interior in two ways for topological vector spaces (when viewed as manifolds).
        However, by \cite[Lemma 1.14]{hg2002a} both conditions yield the same class of mappings. 
        If $U \subseteq \R$ and $g$ is $C^{1}$, we obtain a continuous
        map $g' \colon U \rightarrow TM, g'(x) \coloneq T_x g(1)$. We shall
        write $\frac{\partial}{\partial x}g(x) \coloneq g' (x)$.
 \end{enumerate}
\end{definition}

\begin{definition}\label{defn:
 conno} Let $M$ be a smooth manifold. Then $M$ is called \emph{Banach} (or
 \emph{Fr\'echet}) manifold if all its modelling spaces are Banach (or
 Fr\'echet) spaces. The manifold $M$ is called \emph{locally metrisable} if the
 underlying topological space is locally metrisable (equivalently if all
 modelling spaces of $M$ are metrizable). It is called \emph{metrizable} if it
 is metrisable as a topological space (equivalently locally metrisable and
 paracompact).
\end{definition}

\begin{definition}\label{def:local_addition}
 Suppose $M$ is a smooth manifold. Then a \emph{local addition} on $M$ is a
 smooth map $\A\from U\opn TM\to M$, defined on an open neighbourhood $U$ of
 the submanifold $M\se TM$ such that
 \begin{enumerate}
  \item \label{def:local_addition_a} $\pi\times \A\from U\to M\times M$,
        $v\mapsto (\pi(v),\A(v))$ is a diffeomorphism onto an open
        neighbourhood of the diagonal $\Delta M\se M\times M$ and
  \item \label{def:local_addition_b} $\A(0_{m})=m$ for all $m\in M$.
 \end{enumerate}
 We say that $M$ \emph{admits a local addition} if there exist a local addition
 on $M$.
\end{definition}

\begin{lemma}\label{lem: add:T2}(cf.\
 \cite[10.11]{michor1980}) Suppose that $\A\from U\opn TM\to M$ is a local
 addition on $M$ and that $\tau\from T(TM)\to T(TM)$ is the canonical flip on
 $T(TM)$. Then $ T \A \circ \tau\from \tau(TU)\opn T(TM)\to TM$ is a local
 addition on $TM$. In particular, $TM$ admits a local addition if $M$ does so.
\end{lemma}

\begin{proof}
 Let $0_{M}\from M\to TM$ denote the zero section of $\pi_{M}\from TM\to M$.
 
 The diffeomorphism $\tau\from T(TM)\to T(TM)$ is locally given by
 $(m,x,y,z)\mapsto (m,y,x,z)$ and makes the diagrams
 \begin{equation*}
  \vcenter{  \xymatrix{T(TM)\ar[d]_{T\pi_{M}}\ar[r]^{\tau} & T(TM)\ar[d]^{\pi_{TM}}\\
  TM\ar@{=}[r] & TM
  }}
  \quad\text{ and }\quad
  \vcenter{  \xymatrix{T(TM)\ar[r]^{\tau} & T(TM)\\
  TM\ar@{=}[r]\ar[u]^{T0_{M}} & TM \ar[u]_{0_{TM}}
  }}
 \end{equation*}
 commute \cite[1.19]{michor1980}. Then $\A \circ 0_{M}=\id_{M}$ implies that
 $T \A$ is defined on the open neighbourhood $TU$ of $T0_{M}(TM)$ in $T(TM)$
 and satisfies $T \A \circ T0_{M} =\id_{TM} $. This implies that
 $T \A \circ \tau$ is defined on the open neighbourhood $\tau(T U)$ of
 $0_{TM}(TM)$. It satisfies $T \A \circ \tau \circ 0_{TM}=\id_{TM}$ and thus
 \ref{def:local_addition} \ref{def:local_addition_b} by construction. Moreover,
 if $\pi_{M}\times \A$ is a diffeomorphism from $U$ onto $V\opn M\times M$,
 then $T(\pi_{M}\times \A)=(T \pi_{M}\times T\A)$ is a diffeomorphism from $TU$
 onto $TV\opn T(M\times M)=TM\times TM$. Thus
 $(\pi_{TM}\times T \A \circ \tau)$ is a diffeomorphism from $\tau(TU)$ onto
 $TV$. This establishes \ref{def:local_addition} \ref{def:local_addition_a}.
\end{proof}

\begin{definition}\label{def:smooth_compact_open_topology}
 Let $M,N$ be smooth manifolds. Then we endow the smooth maps $C^{\infty}(M,N)$
 with the initial topology with respect to
 \begin{equation*}
  C^{\infty}(M,N)\hookrightarrow \prod_{k\in\N_{0}}C^0(T^{k}M,T^{k}N)_{c.o.},\quad
  f\mapsto (T^{k}f)_{k\in \N_{0} },
 \end{equation*}
 where $C^0(T^{k}M,T^{k}N)_{c.o.}$ denotes the space of continuous functions
 endowed with the compact-open topology.
\end{definition}

\begin{tabsection}
 From \cite[Proposition 7.3 and Theorem
 5.14]{Wockel13Infinite-dimensional-and-higher-structures-in-differential-geometry}
 we recall the following result.
\end{tabsection}

\begin{theorem}\label{thm:
 sect} Let $E\to M$ be a vector bundle over the compact manifold $M$ such that
 the fibres are locally convex spaces. Then the space of sections
 $\Gamma(M\xleftarrow{}E)$ is a closed subspace of $C^{\infty}(M,E)$ and a
 locally convex space with respect to point-wise addition and scalar
 multiplication. If the fibres of $E\to M$ are metrisable, then so is
 $\Gamma(M\xleftarrow{} E)$ and if the fibres are Fr\'echet spaces, then so is
 $\Gamma(M\xleftarrow{} E)$.
\end{theorem}

\begin{tabsection}
 Our main tool will be the following excerpt from \cite[Theorem
 7.6]{Wockel13Infinite-dimensional-and-higher-structures-in-differential-geometry}.

\end{tabsection}

\begin{theorem}\label{thm:
 MFDMAP} Let $M$ be a compact manifold and $N$ be a locally convex and locally
 metrisable manifold that admits a local addition $\A\from U\opn TN\to N$. Set
 $V:=(\pi\times \A)(U)$, which is an open neighbourhood of the diagonal
 $\Delta N$ in $N\times N$. For each $f\in C^{\infty}(M,N)$ we set
 \begin{equation*}
  O_{f}\coloneq\{g\in C^{\infty}(M,N)\mid (f(x),g(x))\in V \}.
 \end{equation*}
 Then the following assertions hold.
 \begin{enumerate}
  \item \label{thm:manifold_structure_on_smooth_mapping_a} The set $O_{f}$
        contains $f$, is open in $C^{\infty}(M,N)$ and the formula
        $(f(x),g(x))=(f(x),\A(\varphi_{f}(g)(m)))$ determines a homeomorphism
        \begin{equation*}
         \varphi_{f}\from  O_{f}\to \{h\in C^{\infty}(M,TN)\mid \pi(h(x))=f(x)\}\cong \Gamma(f^{*}(TN))
        \end{equation*}
        from $O_{f}$ onto the open subset
        $\{h\in C^{\infty}(M,TN)\mid \pi(h(x))=f(x)\}\cap C^{\infty}(M,U)$ of
        $\Gamma(f^{*}(TN))$.
  \item \label{thm:manifold_structure_on_smooth_mapping_b} The family
        $(\varphi_{f}\from O_{f}\to \varphi_{f}(O_{f}))_{f\in C^{\infty}(M,N)}$
        is an atlas, turning $C^{\infty}(M,N)$ into a smooth locally convex and
        locally metrisable manifold.
  \item \label{thm:manifold_structure_on_smooth_mapping_f} The manifold
        structure on $C^{\infty}(M,N)$ from
        \ref{thm:manifold_structure_on_smooth_mapping_b} is independent of the
        choice of the local addition $\A$.
  \item \label{thm:manifold_structure_on_smooth_mapping_d} If $L$ is another
        locally convex and locally metrisable manifold, then a map
        $f\from L\times M\to N$ is smooth if and only if
        $\wh{f}\from L\to C^{\infty}(M,N)$ is smooth. In other words,
        \begin{equation*}
         C^{\infty}(L\times M,N)\to C^{\infty}(L,C^{\infty}(M,N)), \quad f\mapsto \wh{f}
        \end{equation*}
        is a bijection (which is even natural).
  \item \label{thm:manifold_structure_on_smooth_mapping_h} Let $M'$ be compact
        and $N'$ be locally metrisable such that $N'$ admits a local addition.
        If $\mu\from M'\to M$, $\nu\from N\to N'$ are smooth, then
        \begin{equation*}
         \nu _{*} \mu^{*}\from C^{\infty}(M,N)\to C^{\infty}(M',N'),\quad\gamma\mapsto \nu \circ \gamma \circ \mu
        \end{equation*}
        is smooth.
  \item \label{thm:manifold_structure_on_smooth_mapping_e} If $M'$ is another
        compact manifold, then the composition map
        \begin{equation*}
         \circ\from   C^{\infty}(M',N)\times C^{\infty}(M,M')\to C^{\infty}(M,N), \quad
         (\gamma,\eta)\mapsto \gamma \circ \eta
        \end{equation*}
        is smooth.
 \end{enumerate}
\end{theorem}

\begin{theorem}
 \label{thm:tangent_map_of_pull_back_and_push_forward} Let $M$ be a compact
 manifold and $N$ be a locally convex and locally metrisable manifold that
 admits a local addition. There is an isomorphism of vector bundles
 \begin{equation*}
  \xymatrix@=1em{
  T C^{\infty}(M,N)\ar[dr]_(.4){\pi_{T_{C^{\infty}(M,N)}}} \ar[rr]^{\Phi_{M,N}} && C^{\infty}(M,TN)\ar[dl]^(.425){(\pi_{TN})_{*}}\\
  & C^{\infty}(M,N)
  }
 \end{equation*}
 given by
 \begin{equation*}
  \Phi_{M,N}\from T C^{\infty}(M,N)\to C^{\infty}(M,TN),\quad    \eqclass{t\mapsto \eta(t)}\mapsto
  \left(m\mapsto \eqclass{t\mapsto \eta^\wedge(t,m)}\right).
 \end{equation*}
 Here we have identified tangent vectors in $C^{\infty}(M,N)$ with equivalence
 classes $\eqclass{\eta}$ of smooth curves \\
 $\eta\from \mathopen{]}-\varepsilon,\varepsilon\mathclose{[}\to C^{\infty}(M,N)$
 for some $\varepsilon>0$. The isomorphism $\varphi_{M,N}$ is natural with
 respect to the morphisms from
 \ref{thm:manifold_structure_on_smooth_mapping_h}, i.e., the diagrams
 \begin{equation*}
  \vcenter{         \xymatrix@=1em{
  T C^{\infty}(M,N) \ar[rr]^{\Phi_{M,N}}\ar[d]_{T(\mu^{*})} && C^{\infty}(M,TN)\ar[d]^{\mu_{*}}\\
  T C^{\infty}(M',N) \ar[rr]^{\Phi_{M',N}}\ar[d]_{{T(\nu _{*})}} && C^{\infty}(M',TN)\ar[d]^{({T \nu})_{*}}\\
  T C^{\infty}(M',N') \ar[rr]^{\Phi_{M',N'}} && C^{\infty}(M',TN')
  }}\quad\text{ and }\quad
  \vcenter{         \xymatrix@=1em{
  T C^{\infty}(M,N) \ar[rr]^{\Phi_{M,N}} \ar[d]_{{T(\nu _{*})}} && C^{\infty}(M,TN)\ar[d]^{({T \nu})_{*}}\\
  T C^{\infty}(M,N') \ar[rr]^{\Phi_{M,N'}} \ar[d]_{T(\mu^{*})}&& C^{\infty}(M,TN')\ar[d]^{\mu_{*}}\\
  T C^{\infty}(M',N') \ar[rr]^{\Phi_{M',N'}} && C^{\infty}(M',TN')
  }}
 \end{equation*}
 commute. In particular, $T_{f}C^{\infty}(M,N)$ is naturally isomorphic (as a
 topological vector space) to $\Gamma(f^{*}TN)$ and with respect to this
 isomorphism we have
 \begin{alignat*}{3} 
  T_{f}(\mu^{*})\from & \Gamma(f^{*}TN)\to {\Gamma ((f \circ \mu )^{*}TN)},&&\quad \sigma\mapsto \sigma \circ  \mu \\
  T_{f}(\nu_{*})\from & \Gamma(f^{*}TN)\to {\Gamma ((\nu \circ f)^{*}TN')},&&\quad \sigma\mapsto T \nu \circ \sigma.
 \end{alignat*}
\end{theorem}

\begin{proof}
 First note that $TN$ is also locally convex and locally metrisable and from
 Lemma \ref{lem: add:T2} we infer that it also admits a local addition. 
 Let $\A \colon TN \supseteq \Omega \rightarrow N$ be the local addition on $N$ 
 and $\tau \colon T^2N \rightarrow T^2N$ be the canonical flip (cf.\ Lemma \ref{lem: add:T2}). 
 Then $T\A \circ \tau$ is a local addition on $TN$. 
 Furthermore, $M$ is compact and thus Theorem \ref{thm: MFDMAP}
 implies that $C^\infty (M,N)$, $TC^\infty (M,N)$ and $C^\infty (M,TN)$
 are locally convex manifolds. 
 We can now argue as in \cite[10.12]{michor1980} to see that the charts
 $(\varphi_{0\circ f})_{f\in C^\infty (M,N)}$ cover $C^\infty (M,TN)$. 
 In fact, the charts $(\varphi_{0\circ f})_{f\in C^\infty (M,N)}$ 
 are bundle trivialisations for $(\pi_{TN})_* \colon C^\infty (M,TN) \rightarrow C^\infty (M,N)$ (see \cite[10.12 2. Claim]{michor1980}). 
 The map $\Phi_{M,N}$ will be an isomorphism of vector bundles if we can show that
 it coincides fibre-wise with the isomorphism of vector bundles constructed in the proof of \cite[Theorem 10.13]{michor1980}.
 Note that the proof of \cite[Theorem 10.13]{michor1980} deals only with the case of a finite-dimensional target $N$. 
 However, the local addition constructed in Lemma \ref{lem: add:T2} allows us
 to copy the proof of \cite[Theorem 10.13]{michor1980} almost verbatim\footnote{The changes needed are restrictions of some mappings to open subsets since contrary to \cite[Theorem 10.13]{michor1980} our local additions are not defined on the whole tangent bundle.}.
 To prove that $\Phi_{M,N}$ is indeed of the claimed form, fix $f \in C^\infty (M,N)$. 
 We will evaluate $\varphi_{0\circ f} \circ \Phi_{M,N}$ on the equivalence class $[t\mapsto c(t)]$ of a smooth curve $c \colon ]-\varepsilon, \varepsilon[ \rightarrow C^\infty (M,N)$ with $c(0)=f$:
  \begin{equation} \label{eq: fibrewise} \begin{aligned}
   \varphi_{0\circ f} \circ \Phi_{M,N} ([t\mapsto c(t)]) &= \varphi_{0\circ f} (m\mapsto [t\mapsto c^\wedge (t,m)])\\
							 &= \left(m \mapsto (\pi_{T^2N}, T\A \circ \tau)^{-1} (0\circ f (m), [t\mapsto c^\wedge(t,m)])\right)
   \end{aligned}
  \end{equation}
 By construction we obtain an element in $\Gamma ((0\circ f)^*T^2N) = \Gamma ((0\circ f)^* T^2N|N)$ where $T^2N|N$ is the restriction of the bundle $T^2N$ to the zero-section of $TN$.
 Consider the vertical lift $V_{TN} \colon TN \oplus TN \rightarrow V(TN)$ given locally by $V((x,a),(x,b)) \coloneq (x,a,0,b)$.
 Recall that $\tau$ and $V_{TN}$ are vector bundle isomorphisms.
 Now we argue as in \cite[10.12]{michor1980} to obtain a canonical isomorphism 
  \begin{displaymath}
   I_f \coloneq (f^* (V_{TN})^{-1} \circ f^* \tau)_* \colon \Gamma ((0\circ f)^* T^2N|N) \rightarrow \Gamma (f^*TN) \oplus \Gamma (f^*TN). 
  \end{displaymath}
 (Notice that there is some abuse in notation for $f^*\tau$, explained in detail in \cite[10.12]{michor1980}).
 We will now prove that $I_f$ is the inverse of $\varphi_{0\circ f} \circ \Phi_{M,N} \circ T\varphi_f^{-1}$.
 A computation in canonical coordinates for $T^2N$ yields 
 \begin{equation}\label{eq: local:ident}\begin{aligned}
  T\varphi_f ([t \mapsto c(t)])] &= (m \mapsto [t \mapsto (\pi_{TN}, \A)^{-1} (f (m), c^\wedge (t,m))])) \\
  &= (m \mapsto V_{TN}^{-1} \circ T(\pi_{TN} , \A)^{-1} (0\circ f,[t\mapsto c^\wedge (t,\cdot)])) \in   \Gamma (f^*TN) \oplus  \Gamma (f^*TN) .
  \end{aligned}
 \end{equation}
 Here we have used the identifications $C_f^\infty (M,TN\oplus TN) \cong \Gamma (f^*(TN\oplus TN)) \cong \Gamma (f^*TN) \oplus  \Gamma (f^*TN)$. 
 Since $\tau$ is an involution on $T^2N$ we can compute as follows 
  \begin{equation}\label{eq: composition} \begin{aligned}
   I_f \circ \varphi_{0\circ f} \circ \Phi_{M,N} ([t\mapsto c(t)]) &\stackrel{\eqref{eq: fibrewise}}{=} \left(m \mapsto V_{TN}^{-1} \circ \tau \circ (\pi_{T^2N}, T\A \circ \tau)^{-1} (0\circ f (m), [t\mapsto c^\wedge(t,m)])\right)\\
								   &\stackrel{\hphantom{\eqref{eq: fibrewise}}}{=} \left(m \mapsto V_{TN}^{-1} \circ (\pi_{T^2N} \circ \tau, T\A \circ \tau \circ \tau)^{-1} (0\circ f (m), [t\mapsto c^\wedge(t,m)])\right)  \\
								   &\stackrel{\hphantom{\eqref{eq: fibrewise}}}{=} (m \mapsto V_{TN}^{-1} \circ T(\pi_{TN} , \A)^{-1} (0\circ f,[t\mapsto c^\wedge (t,\cdot)])).
   \end{aligned}
  \end{equation}
Hence the right hand side of \eqref{eq: composition} coincides with the right hand side of \eqref{eq: local:ident}.
Summing up the map $I_f$ is the inverse of $\Phi_{M,N}|_{T_f C^\infty (M,N)}^{C^\infty_{0\circ f} (M,TN)}$. 
We conclude that $\Phi_{M,N}^{-1}$ is the isomorphism of vector bundles described in \cite[Theorem 10.13]{michor1980}. 
The statements concerning the tangent maps of the smooth
 maps discussed in Theorem \ref{thm: MFDMAP}
 \ref{thm:manifold_structure_on_smooth_mapping_h} then follow from
 \cite[Corollary 10.14]{michor1980}.
\end{proof}

\section*{Acknowledgements}

The research on this paper was partially supported by the DFG Research Training
group 1670 \emph{Mathematics inspired by String Theory and Quantum Field
Theory}, the Scientific Network \emph{String Geometry} (DFG project code NI
1458/1-1) and the project \emph{Topology in Norway} (Norwegian Research Council project 213458).

\end{document}